\documentclass[11pt]{amsart}
\usepackage{amssymb,latexsym,graphicx, amscd}
\usepackage{enumerate}

\usepackage{tikz}
\usepackage{caption}
\usetikzlibrary{patterns,arrows,decorations.markings}

\newtheorem{theorem}{Theorem}[section]
\newtheorem{prop}[theorem]{Proposition}
\newtheorem{lemma}[theorem]{Lemma}
\newtheorem{remark}[theorem]{Remark}
\newtheorem{question}[theorem]{Question}
\newtheorem{definition}[theorem]{Definition}
\newtheorem{cor}[theorem]{Corollary}
\newtheorem{example}[theorem]{Example}

\begin{document}

\title{Moduli space of $J$-holomorphic subvarieties}

\author{Weiyi Zhang}
\address{Mathematics Institute\\  University of Warwick\\ Coventry, CV4 7AL, England}
\email{weiyi.zhang@warwick.ac.uk}

\begin{abstract} We study the moduli space of $J$-holomorphic subvarieties in a $4$-dimensional symplectic manifold. For an arbitrary tamed almost complex structure, we show that the moduli space of a sphere class is formed by a family of linear system structures as in algebraic geometry. Among the applications, we show various uniqueness results of $J$-holomorphic subvarieties, {\it e.g.} for the fiber and exceptional classes in irrational ruled surfaces.  On the other hand, non-uniqueness and other exotic phenomena of subvarieties in complex rational surfaces are explored. In particular, connected subvarieties in an exceptional class with higher genus components are constructed. The moduli space of tori is also discussed, and leads to an extension of the elliptic curve theory.
\end{abstract}
 \maketitle

\tableofcontents
\section{Introduction}
In this paper, we study the moduli space of $J$-holomorphic subvarieties where the almost complex structure $J$ is tamed by a symplectic form. Recall $J$ is said to be tamed by a symplectic form $\omega$ if the bilinear form $\omega(\cdot, J(\cdot))$ is positive definite. When we say $J$ is tamed, we mean it is tamed by an arbitrary symplectic form unless it is said otherwise. $J$-holomorphic subvarieties are the analogues of one dimensional subvarieties in algebraic geometry. In our paper, the ambient space $M$ is of dimension four, where subvarieties are just divisors. In \cite{T1}, Taubes provided systematic local analysis of its moduli space $\mathcal M_e$ of $J$-holomorphic subvarieties in a class $e\in H^2(M, \mathbb Z)$ with the Gromov-Hausdorff topology, in particular when the almost complex structure $J$ is chosen generically. For precise definitions and basic properties, see section 2.1.

For an almost complex structure $J$, and a class $e\in H^2(M, \mathbb Z)$, we introduce the $J$-genus of $e$, 
\begin{equation} \label{J-genus}
\begin{array}{lll}
 g_J(e)&=&\frac{1}{2}(e\cdot e+K_J\cdot e)+1,
 \end{array}
 \end{equation}
  where $K_J$ is the canonical class of $J$. A {\it $K_J$-spherical class} (sometimes called sphere class if there is no confusion of choosing a canonical class) is a class $e$ which could be represented by a smoothly embedded sphere and $g_J(e)=0$. An exceptional curve class $E$ is a $K_J$-spherical class such that $E^2=K\cdot E=-1$.\footnote{Since we are in dimension $4$, we will identify an element in $H_2(M, \mathbb Z)$ with its Poincar\'e dual cohomology class by abusing the notation. Usually, we use $e$ to denote a general class in $H^2(M, \mathbb Z)$. The letter $E$ is reserved for an exceptional curve class.} For a generic tamed $J$, any exceptional curve class is represented by a unique embedded $J$-holomorphic sphere with self-intersection $-1$.

For an arbitrary $J$, even it is tamed, the behaviour of reducible $J$-holo\-morphic subvarieties could be very wild. There are even some unexpected phenomenon for a $K_J$-spherical class. For instance, there are classes of exceptional curves, such that the moduli space are of complex dimension $1$ and some representatives have an elliptic curve component. One such example is constructed in \cite{p=h}, recalled in section 6.1. It shows that an exceptional curve class in $\mathbb CP^2\#8\overline{\mathbb CP^2}$ has a $\mathbb CP^1$ family of subvarieties and some of them have an elliptic curve as one of irreducible components. Such examples, although very simple, were not generally expected by symplectic geometers. Since the Gromov-Witten invariant is $1$, people expected to have uniqueness in some sense. This example is extended to all sphere classes in Proposition \ref{T2comp}. This sort of examples could be even wilder. The example constructed above Question 4.18 in \cite{p=h} is disconnected and has a genus $1$ component. In Example \ref{connectedgenus3}, we show the existence of a rational complex surface such that there is a {\it connected} subvariety with a genus $3$ component in an exceptional curve class. Moreover, the graph attached to the subvariety has a loop. This does not contradict to Gromov-Witten theory. 
In fact, none of the subvarieties in a spherical class with higher genus irreducible components contributes to the Gromov-Witten invariant of $e$, see Remark \ref{notGW}.

In \cite{LZrc, LZ-generic}, the notion of $J$-nefness is introduced. A class is said to be {\it $J$-nef} if it pairs non-negatively with all $J$-holomorphic subvarieties. This condition prevents all the exotic phenomena mentioned in the above. Under this assumption, the topological complexity, {\it e.g.} the genus of each irreducible component and the intersection theory, is well controlled. The result is particularly nice when $g_J(e)=0$. In this case, all the irreducible components of subvarieties in class $e$ are rational curves (comparing to Proposition \ref{T2comp} and Example \ref{connectedgenus3}). Moreover, when $e$ is a sphere class with $e\cdot e\ge 0$, we know there is always a smooth $J$-holomorphic curve in class $e$. Both results are sensitive to the nefness condition. In particular, they no longer hold when $e$ is an exceptional curve class in a rational surface as we mentioned above. However, there are no such examples in irrational ruled surfaces. Here, irrational ruled surfaces are smooth $4$-manifolds diffeomorphic to blowups of sphere bundles over Riemann surfaces with positive genus.

\begin{theorem}\label{intro1}
Let $M$ be an irrational ruled surface, and $E$ an exceptional class. Then for any tamed $J$ and any subvariety in class $E$, each irreducible component is a rational curve of negative self-intersection. Moreover, the moduli space $\mathcal M_E$ is a single point.
\end{theorem}

In particular, it confirms Question 4.18 of \cite{p=h} for irrational ruled surfaces.\footnote{Question 4.18 of \cite{p=h} for other symplectic $4$-manifolds is answered affirmatively in the Appendix of \cite{CZ2}.} As other results in this paper, our statement works for an arbitrary tamed almost complex structure, this gives us much more freedom for geometric applications than a generic statement.

The first statement follows from the fact that the positive fiber class of an irrational ruled surface is  $J$-nef for any tamed $J$ (Proposition \ref{smoothT}). Here the positive fiber class is the unique $K_J$-spherical class of square $0$. Then the $J$-nefness technique in \cite{LZrc} gives the desired result. The proof of Proposition \ref{smoothT} requires a new idea. This is based on a simple observation that the adjunction number of a class $e$ is the Seiberg-Witten dimension of $-e$. When the class is not $J$-nef and the $J$-genus of the class is positive, the wall crossing formula of Seiberg-Witten theory would produce non-trivial subvarieties with trivial homology class. To summarize, this observation gives us a strategy to show certain class is $J$-nef.
We expect this observation, along with the nefness technique in \cite{LZrc, LZ-generic}, would lead to more applications. See the discussion in section 3. 

The second statement of Theorem \ref{intro1} follows from a uniqueness result of reducible subvarieties, Lemma \ref{uniquereducible}. This lemma constraints the reducible subvarieties by intersection theory of subvarieties. This is an important ingredient for almost all the results in this paper.

In fact, it follows directly from the second statement of Theorem \ref{intro1} that the $J$-holomorphic subvariety in class $E$ is connected and has no cycle in its underlying graph for any tamed $J$ by Gromov compactness, since these properties hold for the Gromov limit of smooth pseudoholomorphic rational curves.

The nefness of the positive fiber class and Lemma \ref{uniquereducible} also lead to the structure of the moduli space of a sphere class in irrational surfaces for an arbitrary tamed almost complex structure.

\begin{theorem}\label{intro2}
Let $M$ be an irrational ruled surface of base genus $h\ge 1$. Then for any tamed $J$ on $M$, 
\begin{enumerate}
\item there is a unique subvariety in the positive fiber class $T$ passing through a given point;
\item the moduli space $\mathcal M_T$ is homeomorphic to $\Sigma_h$, and there are finitely many reducible varieties;
\item every irreducible rational curve is an irreducible component of a subvariety in class $T$.
\end{enumerate}
\end{theorem}

Theorem \ref{intro1} and Theorem \ref{intro2}(1-2)  hold for generic tamed $J$ on general ruled surfaces regardless they are rational or not. But they hold for arbitrary tamed $J$ only in irrational case. It is likely the following version of Theorem \ref{intro2}(3) is true for general rational surfaces as well: every irreducible negative rational curve is an irreducible component of a subvariety in a sphere class of nonnegative self-intersection. 

In algebraic geometry, Theorem \ref{intro2} could be explained by the linear systems. Recall the long exact sequence 
$$\cdots \longrightarrow H^1(M, \mathcal O)\longrightarrow H^1(M, \mathcal O^*)\stackrel{c_1}{\longrightarrow} H^2(M, \mathbb Z)\longrightarrow \cdots$$
A divisor $D$ gives rise to a line bundle $L_D\in \hbox{Pic}(M)=H^1(M, \mathcal O^*)$. When $M$ is projective, the group of divisor classes modulo linear equivalence is identified with $\hbox{Pic}(M)$. The Poincar\'e-Lelong theorem says that $c_1(L_D)=PD[D]$. In our setting, we fix the class $e\in H^2(M, \mathbb Z)$ (indeed its Poincar\'e dual, but we will not distinguish them in this paper). Any line bundle $L$ with $c_1(L)=e$ would give a projective space family of effective divisors, {\it i.e.} the linear system $(\Gamma(M, L)\setminus \{0\})/ \mathbb C^*$, in the moduli space $\mathcal M_e$. The union of such projective spaces with respect to all possible line bundles with $c_1(L)=e$ is exactly $\mathcal M_e$. Two fibers of an irrational ruled surface are not linearly equivalent, since they are not connected through a family parametrized by rational curves. Hence each projective space is just a point, and the family of these spaces is parametrized by a section of the ruled surface which is diffeomorphic to $\Sigma_h$. In fact, this $\Sigma_h$ is embedded in its Jacobian which is a complex tori $T^{2h}$. Theorem \ref{intro1} could also be interpreted by the linear system, where $\mathcal M_E=\mathbb CP^0$.

When $M$ is simply connected, the long exact sequence implies the uniqueness of the line bundle with given Chern class. Hence the moduli space is always a projective space. It is very interesting to see whether it still holds for a tamed almost complex structure. The following is for rational surfaces.

\begin{theorem}\label{cpl}
Let $J$ be a tamed almost complex structure on a rational surface $M$. Suppose $e$ is a primitive class and represented by a smooth $J$-holomorphic sphere. Then $\mathcal M_e$ is homeomorphic to $\mathbb CP^l$ where $l=\max\{0, e\cdot e+1\}$.
\end{theorem}
In particular, it partially confirms Question 5.25 in \cite{LZ-generic}. Here, $M$ is called a rational surface if it is diffeomorphic to $S^2\times S^2$ or $\mathbb CP^2\#k\overline{\mathbb CP^2}$. We remark that even the connectedness of the moduli spaces $\mathcal M_e$ appearing in Theorems \ref{intro2} and \ref{cpl} was not known.

For the proof of the result, we view $\mathbb CP^l$ as $\hbox{Sym}^lS^2$, the $l$-th symmetric product of $S^2$. There are two main steps in the argument. First we need to find a ``dual" smooth $J$-holomorphic rational curve in a class $e'$ whose pairing with $e$ is $l$. This is achieved by a delicate homological study of $J$-nef classes and techniques from \cite{LZ-generic}. Hence the intersection of elements in $\mathcal M_e$ with this rational curve would give elements of $\hbox{Sym}^lS^2$. Then a refined version of Lemma \ref{uniquereducible} would give us the desired identification.

The only possible non-primitive sphere classes are Cremona equivalent to a double line class in $\mathbb CP^2\#k\overline{\mathbb CP^2}$. Here Cremona equivalence refers to the equivalence under the group of diffeomorphisms preserving the canonical class $K_J$. In this case, we can still show the connectedness of the moduli space and its irreducible part (Proposition \ref{connected}). The connectedness is important in the study of symplectic isotopy problem. More interestingly, a potential generalization of our argument for Theorem \ref{cpl} leads us to a larger framework which generalizes certain part of the elliptic curve theory. In particular, a non-associative (because of the failure of Cayley-Bacharach theorem for a non-integrable almost complex structure) addition is introduced to measure the deviation from the integrability.  

On the other hand, some arguments and techniques in this paper and that of \cite{LZrc, LZ-generic} could be extended to study moduli space of subvarieties in higher genus classes, in particular, tori or classes with $g_J(e)=1$. In this paper, we focus our discussion on the anti-canonical class of $\mathbb CP^2\#8\overline{\mathbb CP^2}$. We are able to show the following:

\begin{theorem}\label{introtori}
If there is an irreducible (singular) nodal curve in $\mathcal M_{-K}$, then $\mathcal M_{smooth, -K}$ and $\mathcal M_{-K}$ are both path connected.
\end{theorem}

We hope to have a more general discussion of $J$-holomorphic tori in future work. 

Section $6$ contains a couple more applications. First, we show that the example mentioned in the beginning is actually a general phenomenon for any non-negative sphere classes. Namely, Proposition \ref{T2comp} says that some subvarieties in a sphere class of a complex surface have an elliptic curve component.  This immediately implies that no sphere class in $\mathbb CP^2\#k\overline{\mathbb CP^2}, k \ge 8$ is $J$-nef for every complex structure $J$. This should be compared with the result mentioned above that the positive fiber class of an irrational ruled surface is $J$-nef for any tamed $J$. 

The other application is on the symplectic isotopy of spheres to a holomorphic curve. This problem is first studied for plane curves, i.e. symplectic surfaces in $\mathbb CP^2$. In this case, the genus of a smooth symplectic surface is totally determined by its degree $d$. It is now known that any symplectic surface in $\mathbb CP^2$ of degree $d\le 17$ is symplectically isotopic to an algebraic curve. Chronologically, for $d=1,2$ (i.e. the sphere case) this result is due to Gromov \cite{Gr}, for $d=3$ to Sikorav \cite{Sik}, for $d\le 6$ to Shevchishin \cite{Shev} and finally $d\le 17$ to Siebert and Tian \cite{ST}.
In Theorem \ref{sphereiso}, we give an alternative proof of the fact (see {\it e.g.} \cite{LW}) that any symplectic sphere $S$ with self-intersection $S\cdot S \ge 0$ in a $4$-manifold $(M, \omega)$ is symplectically isotopic to a holomorphic rational curve.

Besides the techniques of $J$-holomorphic subvarieties, especially the $J$-nefness technique, another important ingredient in our arguments is the Seiberg-Witten theory. In particular, we use SW=Gr and wall-crossing formula frequently. They provide abundant $J$-holomorphic subvarieties when $b^+(M)=1$. As an amusing byproduct, we observe in Proposition \ref{hodge} that the corresponding statement of Hodge conjecture for tamed almost complex structure on $M$ with $b^+(M)=1$ holds. Namely, any element of $H^2(M, \mathbb Z)$ is the cohomology class of a $J$-divisor.

We would like to thank Dmitri Panov for helpful discussion which leads to the paper \cite{LP}, and Fedor Bogomolov for his interest. We are grateful to the referee for careful reading and very helpful suggestions improving the presentation. 
The work is partially supported by EPSRC grant EP/N002601/1.

\section{$J$-holomorphic subvarieties}
In this section, we recall the definition and basic properties of $J$-holomorphic subvarieties. The first two subsections are essentially from \cite{T1, LZ-generic, LZrc}. Then an useful technical lemma on the intersection of $J$-holomorphic subvarieties, Lemma \ref{uniquereducible}, is proved. Finally, after recalling the basics of Seiberg-Witten theory, we show that the almost K\"ahler Hodge conjecture holds when $b^+=1$. 

\subsection{$J$-holomorphic subvarieties}
A closed set $C\subset M$ with finite, nonzero 2-dimensional
Hausdorff measure is said to be an {\it irreducible  $J$-holomorphic subvariety} if it has
no isolated points, and if the complement of a finite set of points
in $C$, called the singular points,  is a connected smooth submanifold with $J$-invariant tangent space. Suppose  $C$ is an irreducible subvariety.
Then it is the image of a $J$-holomorphic map
$\phi:\Sigma\to  M$ from a complex connected curve $\Sigma$, where $\phi$
is an embedding off a finite set. $\Sigma$ is called the model curve and $\phi$ is called the tautological map. 
The map $\phi$ is uniquely determined up to automorphisms of $\Sigma$. 

A {\it $J$-holomorphic subvariety} $\Theta$ is a finite set of pairs $\{(C_i, m_i), 1\leq i\leq n\}$, where each  $C_i$ is irreducible 
$J$-holomorphic subvariety and each $m_i$ is a positive integer. 
The set of pairs is further constrained so that $C_i\ne C_j$ if $i\ne j$. When $J$ is understood, we
will simply call a  $J$-holomorphic subvariety a subvariety. 
They are the analogues of one dimensional subvarieties in algebraic geometry. Taubes provides a  systematic analysis of pseudo-holomorphic subvarieties in \cite{T1}.

A subvariety $\Theta=\{(C_i, m_i)\}$ is said to be connected if $\cup C_i$ is connected. We call $\Theta>\Theta_0$ if $\Theta-\Theta_0$ is another, possibly empty, subvariety.

The associated homology class $e_C$ (sometimes, we will also write it by $[C]$) is defined to be the push forward of the fundamental class of $\Sigma$ via $\phi$. 
And for a subvariety $\Theta$, the associated class $e_{\Theta}$ is defined to be $\sum m_ie_{C_i}$.

An irreducible  subvariety  is said to be {\it smooth} if it has no singular points. 
A special feature in dimension 4 is that,  by the adjunction formula,   the genus of a smooth subvariety $C$ is given  by $g_J(e_C)$. 
For a general class $e$ in $H^2(M;\mathbb Z)$, recall the $J$-genus of $e$ is defined by
$$
 g_J(e)=\frac{1}{2}(e\cdot e+K_J\cdot e)+1
$$
  where $K_J$ is the canonical class of $J$. In general, $g_J(e)$ could take any integer value. Let $\mathcal J^{\omega}$ be the space of $\omega$-tamed almost complex structures. Notice the $J$-genus is an invariant for $J\in \mathcal J^{\omega}$ since $\mathcal J^{\omega}$ is path connected and $K_J$ is invariant under deformation. Hence, later we will sometimes write $g_{\omega}(e)=g_J(e)$ when a symplectic structure $\omega$ is fixed.

Moreover, when $C$ is an irreducible subvariety,  $g_J(e_C)$ is non-negative.  In fact, by the adjunction inequality in \cite{McD},  $g_J(e_C)$ is 
bounded from below by the genus of the  model curve $\Sigma$ of $C$, with equality if and only if $C$ is smooth. Especially, when $g_J(e_C)=0$, $C$ is a smooth rational curve.

An element $\Theta$, in the moduli space $\mathcal M_e$ of subvarieties in the class $e$, is a 
subvariety with  $ e_{\Theta}=e$. $\mathcal M_e$ has a natural topology in the following Gromov-Hausdorff sense. 
 Let $|\Theta|=\cup_{(C, m)\in \Theta}C$ denote the support of $\Theta$. Consider  the symmetric, non-negative function, $\varrho$, on $\mathcal M_e\times \mathcal M_e$ that is defined by the following rule:
\begin{equation} \varrho(\Theta, \Theta')=\sup _{z\in |\Theta|} \hbox{dist}(z, |\Theta'|)
+\sup _{z'\in |\Theta'|} \hbox{dist}(z', |\Theta|).
\end{equation}
The function $\varrho$ is used to measure distances on $\mathcal M_e$, where the distance function dist($\cdot, \cdot$) is defined by an almost Hermitian metric on $(M, J)$.

Given a smooth $2$-form $\nu$ we introduce the pairing
$$(\nu, {\Theta} )=\sum_{(C, m)\in \Theta} m\int_{C}\nu.$$

The topology on $\mathcal M_e$ is defined in terms of convergent sequences:

A sequence $\{\Theta_k\}$ in $\mathcal M_e$ converges to a given element $\Theta$ if the following two conditions are met:

\begin{itemize}
\item  $\lim_{k\to \infty} \varrho (\Theta, \Theta_k)=0$.

\item  $\lim_{k\to \infty} (\nu, \Theta_k)=(\nu, \Theta)$ for any given smooth 2-form $\nu$.
\end{itemize}

That the moduli space $\mathcal M_e$ is compact is an application of Gromov compactness, see Proposition 3.1 of \cite{T1}.

\begin{definition}\label{eff}
A  homology class $e\in H_2(M; \mathbb Z)$ is said to be  
$J$-effective if $\mathcal M_e$ is nonempty. 
\end{definition}

We use $\mathcal M_{irr, e}$  to denote the moduli space of irreducible subvarieties in class $e$. Let $\mathcal M_{red,e}:=\mathcal M_e \setminus \mathcal M_{irr,e}$. 

Given a class $e$, its $J$-dimension is
\begin{equation} \label{l}  \iota_e=\frac{1}{2}(e\cdot e-K_{J}\cdot e). 
\end{equation}

The integer $\iota_e$ is the expected (complex) dimension of the moduli space $\mathcal M_e$. When $g_J(e)=0$, we have $\iota_e=e\cdot e+1$.
When $e$ is a class  represented by a smooth  rational curve ({\it i.e.} $J$-holomorphic sphere),
we introduce 
$$l_e=\max\{\iota_e, 0\}.$$

Given a $k(\le {l_e})$-tuple of distinct points $\Omega$, recall that $\mathcal M_e^\Omega$ is the space of subvarieties in $\mathcal M_e$  that contains all entries of $\Omega$. Introduce similarly $\mathcal M_{irr,e}^{\Omega}$ and $\mathcal M_{red,e}^{\Omega}$.  We will often drop the subscript $e$ when there is no confusion.

\subsection{$J$-nef classes}\label{secnef}
In general, all these moduli spaces could behave wildly. The notion of $J$-nefness provides good control as shown in \cite{LZrc, LZ-generic}. 

A class $e$ is said to be {\it $J$-nef} if it pairs  non-negatively with any $J$-holomorphic subvariety. When there is a $J$-holomorphic subvariety in a $J$-nef class $e$, {\it i.e.} $e$ is also effective, we have $e\cdot e\ge 0$. A $J$-nef class $e$ is said to be {\it big} if $e\cdot e>0$. The vanishing locus $Z(e)$ of a big $J$-nef class $e$  is the union of irreducible subvarieties $D_i$ such that $e\cdot e_{D_i}=0$. Denote  the complement of the vanishing locus of $e$  by $M(e)$. From the definition and  the positivity of intersections of distinct irreducible subvarieties \cite{Gr, MW}, it is clear that there does not exist an irreducible subvariety in class $e$ passing through $x\in Z(e)$ when $e$ is big and $J$-nef.

If the support $|C|=\cup C_i$ of subvariety $\Theta=\{(C_i, m_i)\}$ is connected, then Theorem 1.4 of \cite{LZrc} says that  \begin{equation}\label{genus bound} g_J(e)\ge \sum_i g_J(e_{C_i})
 \end{equation}
for a $J$-nef class $e$ with $g_J(e)\ge 0$. In this paper, we use the following result which follows from the above genus bound and is read from Theorem 1.5 of \cite{LZrc}.
\begin{theorem}\label{spheresphere}
Suppose $J$ is tamed by some symplectic structure, e is a $J$-nef class with $g_J(e)=0$ and $\Theta\in \mathcal M_e$. Then $\Theta$ is connected and each irreducible component  of $\Theta$  is a smooth rational curve.
\end{theorem}

Moreover, when $e$ is $J$-nef and $J$-effective with $g_J(e)=0$, we have the following strong bound for the expected dimension of curve configuration for $\Theta\in \mathcal M_{red, e}$ (Lemma 4.10 in \cite{LZrc})
\begin{equation}\label{red-dim}
\sum_{(C_i, m_i)\in \Theta}l_{e_i}\leq\sum_{(C_i, m_i)\in \Theta} m_i l_{e_i}\leq l_e-1.
\end{equation}

Along with automatic transversality, we have the following which is extracted from Proposition 4.5 and Proposition 4.10 of \cite{LZ-generic}.
\begin{theorem}\label{unobstructed}
Suppose $e$ is a $J$-nef spherical class with $e\cdot e\ge 0$. Then $\mathcal M_{irr, e}$ is a non-empty smooth manifold of dimension $2l_e$ and $\mathcal M_{red, e}$ is a finite union of compact manifolds, each with dimension at most $2(l_e-1)$.
\end{theorem}

This is an unobstructedness result for the deformation of symplectic surfaces. In \cite{Sik}, an unobstructed result is obtained. In our circumstance, it implies that when $\mathcal M_{irr, e}\neq \emptyset$, it is a smooth manifold. Hence, our main contribution is to show $\mathcal M_{irr, e}\neq \emptyset$ when $e$ is $J$-nef. It is important for our applications, since we will deform $J$ in $\mathcal J^{\omega}$ and the irreducible part of moduli space need not to be nonempty {\it a priori}. Our result for $\mathcal M_{red, e}$ is more general since \cite{Sik} need each component of $\Theta$ has multiplicity one and has self-intersection no less than $-1$.

\subsection{Intersection of subvarieties}\label{intersection}
We first analyze how an intersection point contributes to the intersection number of two subvarieties. Since every component of a subvariety is an irreducible curve, the intersection number will always contribute positively. 

There are two typical types of intersections. The first is when two multiple components $(C, n)$ and $(C', m)$ have an intersection point $p$. If the two irreducible curves $C$ and $C'$ intersect at $p$ transversally, then the point $p$ contributes $mn$ to the intersection numbers. The second type is when two curves $C$ and $C'$ have high contact order at $p$. If they are tangent to each other at order $n$, which means the local Taylor expansion coincides up to order $n-1$, then $p$ would contribute $n$ to the intersection number. Notice only the local behavior of the two curves matters for the intersection near $p$. Hence the two types could interact simultaneously. 

\begin{example}
Suppose $\Theta$ is a subvariety with two irreducible components $(C_1,m_1)$ and $(C_2, m_2)$, which intersect transversally at point $p$, and $\Theta'$  is another subvariety with a component $(C', m)$, passing through $p$ and tangent to $C_1$ of order $n$ at $p$. The point $p$ would contribute $nmm_1+mm_2$ to the intersection of two subvarieties $\Theta$ and $\Theta'$. 
\end{example}

Later in this paper, we will see in several occasions to prescribe a subvariety passing through given ``points with weight", which will be explained immediately. Corresponding to the above two types of intersections of subvarieties, there are two types of points with weight. The first type, denote by $(x, d)$ with $x\in M$ and $d\in \mathbb Z$, means the subvariety $\Theta$ passes through point $x$ with multiplicity $d$. Since no direction or higher order contact is given, the multiplicity here is the sum of weights of all irreducible components of $\Theta$ passing through $x$, say $(C_1, m_1), \cdots, (C_k, m_k)$, {\it i.e.} $d=m_1+\cdots+m_k$. 

The second type, denote by $(x, C, d)$ with  $x\in M$, $d\in \mathbb Z$ and $C$ a (local) $J$-holomorphic curve passing through $x$, means subvariety $\Theta$ passes through point $x$ with multiplicity $d$ counted with contact orders with $C$. Precisely, if locally there are local components of $\Theta$, say $(C_1, m_1), \cdots, (C_k, m_k)$ passing through point $x$ and tangent to the curve $C$ with order $d_1, \cdots, d_k$ respectively, then $d=d_1m_1+\cdots +d_km_k$. Here we implicitly assume $C$ is of multiplicity one. In the most general case, we consider $(C, n)$, and the corresponding relation is $d=n(d_1m_1+\cdots +d_km_k)$. Sometimes, we call $C$ the ``matching" curve at point $x$.

The following strengthens Lemma 4.18 in \cite{LZ-generic}, considering the first type intersection. 

\begin{lemma}\label{uniquereducible}
Let $J$ be an almost complex structure on $M^4$. Suppose $e$ is $J$-nef with 
$l=\max\{e\cdot e+1, 0\}$ and $\{(x_1, d_1), \cdots, (x_k, d_k)\}$ are points with weight. 
\begin{enumerate}
\item Suppose two subvarieties $\Theta, \Theta'\in \mathcal M_e$ do not share irreducible components. If they both pass through these points with weight, then $d_1+\cdots+d_k< l$.
\item Let $\Theta=\{(C_i, m_i)\}\in \mathcal M_e$ be a connected subvariety passing through these points with weight such that there are at least $m_ie\cdot e_{C_i}$ points (counted with multiplicities) on $C_i$ for each $i$ and all $x_i$ are smooth points. Then there is no other such subvariety in class $e$ that shares an irreducible component with $\Theta$. 
\end{enumerate}

\end{lemma}

\begin{proof}
The first statement simply follows from positivity of intersection of two distinct irreducible $J$-holomorphic curves. This is because the points $(x_i, d_i)$ are in the intersection of $\Theta$ and $\Theta'$ and each $d_i$ is no greater than the local intersection index of them at $x_i$. These local intersection indices are positive integers which add up to $e\cdot e$, although there might be  intersection points of $\Theta$ and $\Theta'$ other than $x_i$. Thus, the inequality follows.  

For the second statement, suppose there is another such subvariety $\Theta'$, such that $\Theta$ and $\Theta'$ share at least one common irreducible components. 

We rewrite two subvarieties $\Theta, \Theta' \in \mathcal M_e$, allowing $m_i=0$ in the notation, such that they share the same set of irreducible components formally, i.e. $\Theta=\{(C_i, m_i)\}$ and $\Theta'=\{(C_i, m'_i)\}$. Then for each $C_i$, if $m_i\le m'_i$, we change the components to $(C_i, 0)$ and $(C_i, m'_i-m_i)$. At the same time, if a point $x$, as one of $x_1, \cdots, x_k$, is on $C_i$, then the weight is reduced by $m_i$ as well. Similar procedure applies to the case when $m_i> m_i'$.  Apply this process to all $i$ and discard finally all components with multiplicity $0$ and denote them by $\Theta_0,\Theta'_0$ and still use $(C_i, m_i)$ and $(C_i, m'_i)$ to denote their components. Notice they are homologous, formally having homology class $$e-\sum_{m_{k_i}<m_{k_i}'} m_{k_i}e_{C_{k_i}}-\sum_{m_{l_j}'<m_{l_j}} m'_{l_j}e_{C_{l_j}}-\sum_{m_{q_p}'=m_{q_p}} m'_{q_p}e_{C_{q_p}}.$$ 

There are two ways to express the class, by taking $e=e_{\Theta}$ or $e=e_{\Theta'}$ in the above formula. Namely, it is $$\sum_{m_{k_i}<m_{k_i}'}(m_{k_i}'-m_{k_i})e_{C_{k_i}}+\hbox{others}=e_{\Theta_0'}=e_{\Theta_0}=\sum_{m_{l_j}'<m_{l_j}} (m_{l_j}-m'_{l_j})e_{C_{l_j}}+\hbox{others}.$$
Here the term ``others" means the terms $m_ie_{C_i}$ or $m_i'e_{C_i}$ where $i$ is not taken from $k_i$, $l_j$ or $q_p$. 

Now $\Theta_0$ and $\Theta_0'$ have no common components. By the process we just applied, counted with weight, there are  at least $e\cdot e_{\Theta_0}$ points on $\Theta_0$. These points are also contained in $\Theta_0'$ with right weights. Hence $\Theta_0$ and $\Theta_0'$ would intersect at least $e\cdot e_{\Theta_0}$ points with weight.

We notice that $e\cdot e_{\Theta_0}\ge e_{\Theta_0}\cdot e_{\Theta_0'}$. In fact, the difference $e-e_{\Theta_0}=e-e_{\Theta_0'}$ has $3$ types of terms, any of them pairing non-negatively with the class $e_{\Theta_0}$. For the terms with index $k_i$, {\it i.e.} the terms with $m_{k_i}<m_{k_i}'$, we use the expression of $e_{\Theta_0}=\sum_{m_{l_j}'<m_{l_j}} (m_{l_j}-m'_{l_j})e_{C_{l_j}}+\hbox{others}$ to pair with. Since the irreducible curves involved in the expression are all different from $C_{k_i}$, we have $e_{C_{k_i}}\cdot e_{\Theta_0}\ge 0$. Similarly, for $C_{l_j}$, we use the expression of $e_{\Theta_0'}=\sum_{m_{k_i}<m_{k_i}'}(m_{k_i}'-m_{k_i})e_{C_{k_i}}+\hbox{others}$. We have  $e_{C_{l_j}}\cdot e_{\Theta_0'}\ge 0$. For $C_{q_p}$, we could use either $e_{\Theta_0}$ or $e_{\Theta_0'}$. Since $e_{\Theta_0}=e_{\Theta_0'}$, we have $(e-e_{\Theta_0})\cdot e_{\Theta_0}\ge 0$. 

Moreover, we have the strict inequality $e\cdot e_{\Theta_0}> e_{\Theta_0}^2$. This is because we assume the original $\Theta, \Theta'$  are connected and have at least one common component. The first fact implies there is at least one index in $k_i$, $l_j$ or $q_p$. The second fact implies at least one of the intersection of $C_{k_i}$, $C_{l_j}$ or $C_{q_p}$ with $e_{\Theta_0}$ as in the last paragraph would take positive value.

As we have shown that $\Theta_0$ and $\Theta_0'$ would intersect at least $e\cdot e_{\Theta_0}$ points with weight, the inequality  $e\cdot e_{\Theta_0}> e_{\Theta_0}^2$ implies 
the sum of local intersection indices of $\Theta_0$ and $\Theta_0'$ is greater than the homology intersection number $e_{\Theta_0}^2$ of our new subvarieties $\Theta_0$ and $\Theta_0'$. This contradicts to the local positivity of intersection and the fact that $\Theta_0, \Theta_0'$ have no common component. The contradiction implies that $\Theta$ is the unique such subvariety as described in the statement. 
\end{proof}

The lemma and its argument will be used later, in particular, Theorem \ref{red-1}, Theorem \ref{M_T} and Proposition \ref{connected}.
A similar statement for the more general second type intersection will be proved by a similar argument and used in Theorem \ref{homeoCPl}.

\subsection{Seiberg-Witten invariants and subvarieties}\label{SWwc}
Other than techniques in \cite{LZ-generic, LZrc}, another important ingredient of our method is the Seiberg-Witten invariant.  We follow the notation in \cite{p=h}. However we need a more general setting.

Let $M$ be an oriented $4$-manifold with a given Riemannian metric $g$ and a spin$^{c}$ structure $\mathcal L$.  Hence there are a pair of rank $2$ complex vector bundles $S^{\pm}$ with isomorphisms $\det(S^+)=\det(S^-)=\mathcal L$. The Seiberg-Witten equations are for a pair $(A, \phi)$ where $A$ is a connection of $\mathcal L$ and $\phi\in \Gamma(S^+)$ is a section of $S^+$. These equations are 
$$D_A\phi=0$$
$$F_A^+=iq(\phi)+i\eta$$ where $q$ is a canonical map $q: \Gamma(S^+)\rightarrow \Omega^2_+(M)$ and $\eta$ is a self-dual $2$-form on $M$. 

The group $C^{\infty}(M; S^1)$ naturally acts on the space of solutions. Under this action, the map $f\in C^{\infty}(M; S^1)$ sends a pair $(A, \phi)$ to $(A+2fdf^{-1}, f\phi)$. It acts freely at irreducible solutions. Recall a reducible solution has $\phi=0$, and hence $F_A^+=i\eta$. The quotient is the moduli space, denoted by $\mathcal M_M(\mathcal L, g, \eta)$. For generic pairs $(g, \eta)$, the Seiberg-Witten moduli space $\mathcal M_M(\mathcal L, g, \eta)$ is a compact manifold of dimension $$2d(\mathcal L)=\frac{1}{4}(c_1(\mathcal L)^2-(3\sigma(M)+2\chi(M)))$$ where $\sigma(M)$ is the signature and $\chi(M)$ is the Euler number. Furthermore, an orientation is given to $\mathcal M_M(\mathcal L, g, \eta)$ by fixing a homology orientation for $M$, {\it i.e.} an orientation of $H^1(M)\oplus H^2_+(M)$. When $b^+(M)=1$, the space of $g$-self-dual forms $\mathcal H^+_g(M)$ is spanned by a single harmonic $2$-form $\omega_g$ of norm $1$ agreeing with the homology orientation.

Quotient out the space of triple $(p, (A, \phi))$ where $p\in M$ and $(A, \phi)$ is a solution of Seiberg-Witten equation by based actions $f\in C^{\infty}(M; S^1)$ with $f(p)=1$, we obtain a smooth manifold $\mathcal E$. It is a principal $S^1$ bundle over $M\times  \mathcal M_M(\mathcal L, g, \eta)$. The slant product with $c_1(\mathcal E)$ defines a natural map $\psi$ from $H_*(M, \mathbb Z)$ to $H^{2-*}(\mathcal M_M(\mathcal L, g, \eta), \mathbb Z)$.

We now assume $(M, J)$ is an almost complex $4$-manifold with canonical class $K$. We denote $e:=\frac{c_1(\mathcal L)+K}{2}\in H^2(M; \mathbb Z)/(2\hbox{-torsion})$. For a generic choice of $(g, \eta)$, the Seiberg-Witten invariant $SW^*_{M, g, \eta}(e)$ takes value in $\Lambda^*H^1(M, \mathbb Z)$. If $d(\mathcal L)<0$, then the SW invariant is defined to be zero. Otherwise, let $\gamma_1\wedge \cdots\wedge \gamma_p\in \Lambda^p(H_1(M, \mathbb Z)/\hbox{Torsion})$, we define 
\begin{equation}\label{SWdef}
SW^*_{M, g, \eta}(e; \gamma_1\wedge\cdots\wedge \gamma_p):=\int_{\mathcal M_M(\mathcal L, g, \eta)} \psi(\gamma_1)\wedge\cdots\wedge\psi(\gamma_p)\wedge\psi(pt)^{d-\frac{p}{2}}.
\end{equation}

If $b^+>1$, a generic path of $(g, \eta)$ contains no reducible solutions. Hence, the Seiberg-Witten invariant is an oriented diffeomorphism invariant in this case. Hence we can use the notation $SW^*(e)$ for the (full) Seiberg-Witten invariant. We will also write $$\dim_{SW}(e)=2d(\mathcal L)=e^2-K\cdot e$$ for the Seiberg-Witten dimension. In the case $b^+=1$, there might be reducible solutions on a $1$-dimensional family.  Recall that the curvature $F_A$ represents the cohomology class $-2\pi ic_1(\mathcal L)$. Hence $F_A^+=i\eta$ holds only if $-2\pi c_1(\mathcal L)^+=\eta$. This happens if and only if the discriminant $\Delta_{\mathcal L}(g, \eta):=\int (2\pi c_1(\mathcal L)+\eta)\omega_g=0$. With this in mind, the set of pairs $(g, \eta)$ with positive (resp. negative) discriminant is called the positive (resp. negative) $\mathcal L$ chamber. We use the notation $SW^*_{\pm}(e)$ for the Seiberg-Witten invariants in these two chambers. The part of $SW^*(e)$ (resp. $SW_{\pm}^*(e)$) in $\Lambda^{i}H^1(M, \mathbb Z)$ will be denoted by $SW^i(e)$ (resp. $SW_{\pm}^i(e)$). Moreover, in the this paper, we will use $SW^*(e)$ instead of $SW^*_-(e)$ when $b^+=1$. For simplicity, the notation $SW(e)$ is reserved for $SW^0(e)$.

We now assume $(M, \omega)$ is a symplectic $4$-manifold, and $J$ is a $\omega$-tamed almost complex structure. Then the results in \cite{T, LLb+1} equate Seiberg-Witten invariants with Gromov-Taubes invariants that are defined by making a suitable counting of $J$-holomorphic subvarieties. In fact, our $SW^*(e)$ used in this paper is essentially the Gromov-Taubes invariant in the literature. In particular, our $SW^*(e)$ is the original Seiberg-Witten invariant of the class $2e-K$. The key conclusion we will take from this equivalence is that when $SW^*(e)\ne 0$, there is a $J$-holomorphic subvariety in class $e$. 
Moreover, if $SW(e)\ne 0$, there is a $J$-holomorphic subvariety in class $e$ passing through $\dim_{SW}(e)$ given points. 

Hence, to produce subvarieties in a given class, we will prove nonvanishing results for $SW^*(e)$, usually for $SW(e)$. When $b^+(M)>1$, an important result of Taubes says that $SW(K)=1$. When $b^+(M)=1$, the key tool is the wall-crossing formula, which relates the Seiberg-Witten invariants of classes $K-e$ and $e$ when  $\dim_{SW}(e)\ge 0$. The general wall-crossing formula is proved in \cite{LLwall}. In particular, when $M$ is rational or ruled, we have $$|SW(K-[C])-SW([C])|=\begin{cases}   1 & \hbox{if $(M, \omega)$ rational,}\cr
  |1+[C]\cdot T|^h &\hbox{if
  $(M, \omega)$ irrationally ruled,}\cr\end{cases}$$
  where $T$ is the unique positive fiber class and $h$ is the genus of base surface of irrationally ruled manifolds. For a general symplectic $4$-manifold with $b^+(M)=1$, usually the wall-crossing number for $SW(e)$ is hard to determine and sometimes vanishes \cite{LLwall}. However, we still have a simple formula for top degree part of Seiberg-Witten invariant (see Lemma 3.3 (1) of \cite{LL}). 

\begin{prop}\label{generalWC}
Let $M$ be a symplectic $4$-manifold with $b^+=1$ and canonical class $K$. 
Suppose $\dim_{SW}(e)\ge b_1$. Let $\gamma_1, \cdots, \gamma_{b_1}$ be a basis of $H_1(M, \mathbb Z)/\hbox{Torsion}$ such that $\gamma_1\wedge\cdots\wedge \gamma_{b_1}$ is the dual orientation of that on $\Lambda^{b_1}(H^1(M, \mathbb Z))$. Then $$|SW^{b_1}(K-e; \gamma_1\wedge\cdots\wedge \gamma_{b_1})-SW^{b_1}(e; \gamma_1\wedge\cdots\wedge \gamma_{b_1})|=1.$$
\end{prop}

Here, $b_1$ stands for the first Betti number. In particular, it implies a nonvanishing result: let $e\in H^2(M, \mathbb Z)$ be a class with $e^2\ge 0$, $K\cdot e\le 0$, and at least one of the inequalities being strict, then $SW^*(ke)\ne 0$ for sufficiently large $k$.

\subsection{Almost K\"ahler Hodge conjecture}
Let $X$ be a non-singular complex projective manifold. The (integral) Hodge conjecture asks whether every class in $H^{2k}(X, \mathbb Q)\cap H^{k, k}(X)$ (resp. $H^{2k}(X, \mathbb Z)\cap H^{k, k}(X)$) is a linear combination with rational (resp. integral) coefficients of the cohomology classes of complex subvarieties of $X$. When $\dim_{\mathbb C}X\le 3$, Hodge conjecture is known to be true and follows from  Lefschetz theorem on $(1, 1)$ classes. The integral Hodge conjecture, which was Hodge's original conjecture, is known to be false for some projective $3$-folds. 

In this subsection, we will show an amusing result, which basically says that the integral Hodge conjecture, or Lefschetz theorem on $(1, 1)$ classes, is true for almost K\"ahler $4$-manifolds of $b^+=1$.

It is well known that in general the almost K\"ahler Hodge conjecture statement is not true if $b^+>1$, even when our manifold is K\"ahler. The most well known counterexample is a generic CM complex tori. It has no subvarieties in general, but the group of integral Hodge classes has $\dim H^{1,1}(M, \mathbb Z)=2$. See the appendix of \cite{zuc}.

In our situation, $H_J^+(M)\cap H^2(M, \mathbb K)$ plays the role of $H^{1,1}(M, \mathbb K)$ for $\mathbb K=\mathbb Z$ or $\mathbb Q$. Here $H_J^+(M)$ is called the $J$-invariant cohomology which is introduced in \cite{LZ, DLZ} along with the $J$-anti-invariant $H_J^-(M)$. Recall that an almost complex
structure acts on the bundle of real 2-forms $\Lambda^2$ as an
involution, by $\alpha(\cdot, \cdot) \rightarrow \alpha(J\cdot,
J\cdot)$. This involution induces the splitting into $J$-invariant, respectively,
$J$-anti-invariant 2-forms
$\Lambda^2=\Lambda_J^+\oplus \Lambda_J^-$. Then we define 
$H_J^{\pm}(M)=\{ \mathfrak{a} \in H^2(M;\mathbb R) | \exists \;
\alpha\in  \Lambda_J^{\pm}, \, d\alpha=0 \mbox{ such that } [\alpha] =
\mathfrak{a}\}$.

A {\it divisor} (resp. {\it $\mathbb Q$-divisor}) with respect to an almost complex structure $J$ is
 a finite formal sum $\sum a_iC_i$ where $C_i$ are $J$-holomorphic irreducible curves and $a_i\in \mathbb Z$ (resp. $a_i\in \mathbb Q$).

\begin{prop}\label{hodge}
Let $M$ be a symplectic $4$-manifold with $b^+(M)=1$, and $J$ a tamed almost complex structure on it. Any element of $H^2(M, \mathbb Z)$ is the cohomology class of a divisor (with respect to $J$). 
\end{prop}
\begin{proof}
When $b^+(M)=1$, by Corollary 3.4 of \cite{DLZ}, we have $h_J^-=\dim H_J^-(M)=0$ and $H_J^+(M)=H^2(M, \mathbb R)$. Let $e_1, \cdots, e_{b_2}$ be a $\mathbb Z$-basis of $H^2(M, \mathbb Z)$, and $\alpha_1, \cdots, \alpha_{b_2}$ $2$-forms representing them. Since being a $J$-tamed symplectic form is an open condition, if $J$ is tamed by a symplectic form $\omega$, we can choose $\omega$ such that $[\omega]\in H^2(M, \mathbb Q)$. Then we can find a large integer $N$ and $b_2+1$ $J$-tamed symplectic forms $\omega_i=N\omega+\alpha_i$ with $[\omega_i]=N[\omega]+e_i\in H^2(M, \mathbb Z)$ when $1\le i\le b_2$ and $\omega_0=N\omega$. Their cohomology classes generate the vector space $H^2(M, \mathbb Z)$. 

If we choose $L>k:=\max_i\{0, \frac{K\cdot [\omega_i]}{[\omega_i]\cdot [\omega_i]}\}+b_1$, we have $$\dim_{SW}(L[\omega_i])=L(L[\omega_i]^2-K\cdot [\omega_i])> L((K\cdot [\omega_i]+b_1)-K\cdot[\omega_i]) \ge b_1.$$ 

Apply Proposition \ref{generalWC}, we have $SW^{b_1}(L[\omega_i])\ne SW^{b_1}(K-L[\omega_i])$. We claim that when  $L>k$, $SW^{b_1}(L[\omega_i])\ne 0$ for any $i$. By wall-crossing, we only need to show that $SW^{b_1}(K-L[\omega_i])= 0$. We prove it by contradiction. If $SW^*(K-L[\omega_i])\ne 0$, then $K-L[\omega_i]$ will be the class of a $J$-holomorphic subvariety and hence an $\omega_i$-symplectic submanifold. However, when $L>k$, we have $(K-L[\omega_i])\cdot [\omega_i]<0$, which is a contradiction. Hence, we have $SW(L[\omega_i])\ne 0$ for $L>k$ and  there are subvarieties in class $L[\omega_i]$ for any $i$.

Let $a\in H^2(M, \mathbb Z)$ be an arbitrary class. Because of the way we choose our $\omega_i$, we have $a=\sum_{i=0}^{b_2}a_i[\omega_i]$ with $a_i\in \mathbb Z$. Now we further write it as $a=\sum_{i=0}^{b_2}a_i(L+1)[\omega_i]-\sum_{i=0}^{b_2}a_i L[\omega_i]$, which implies  $a$ is the cohomology class of a divisor.
\end{proof}

\begin{remark}
There is another argument to prove $SW(K-L[\omega_i])= 0$ for large $L$. This is because $K-L[\omega_i]$ pairs negatively with $2K$ for non-rational or non-ruled manifolds, with $H$ for $\mathbb CP^2\# k\overline{\mathbb CP^2}$, with a positive fiber class $A$ for $S^2\times S^2$, and with the positive fiber class $T$ for irrational ruled manifolds. All of the classes mentioned above are SW non-trivial classes with a representative of irreducible $J$-holomorphic non-negative self-intersections. Hence the contradiction follows from Lemma \ref{int>0}  by taking $e=K-L[\omega_i]$.
\end{remark}

We remark that the symplectic version of Hodge conjecture holds for any compact symplectic manifolds $(M^{2n}, \omega)$. More precisely, in \cite{hvle}, it shows that any element of $H_{2k}(M^{2n}, \mathbb Z)$ is a symplectic $\mathbb Q$-cycle in the form $\frac{1}{N}[S_1^{2k}]-\frac{1}{N}[S_2^{2k}]$ where $N$ is a positive integer and $S_i^{2k}$ are symplectic submanifolds of dimension $2k$.

\section{Irrational ruled surfaces}
In this section, we use the techniques of \cite{LZrc, LZ-generic} along with Seiberg-Witten theory to identify the moduli space of $J$-holomorphic subvarieties in the fiber class of irrational ruled surfaces for any tamed almost complex structure $J$. When the irrational ruled surface is minimal, it was handled by McDuff in a series of papers, in particular \cite{Mc2}. For non-minimal irrational ruled surfaces, the structure of reducible subvarieties was not clear for a non-generic tamed almost complex structure. The work of \cite{LZ-generic, LZrc} developed a toolbox to study this kind of problems.

To apply the results and techniques from \cite{LZ-generic, LZrc}, one has to check the $J$-nefness of the classes we are dealing with. For previous applications, like Nakai-Moishezon type duality and the tamed to compatible question, we could always start with a $J$-nef class. However, for most other applications like our  problem in this section, we do not know $J$-nefness {\it a priori}.   In the following, we will develop a strategy to verify this technical condition. Then along with the techniques in \cite{LZrc, LZ-generic}, we cook up a general scheme to investigate the moduli space of subvarieties (and its irreducible and reducible parts) in a given class.

The following lemma is Lemma 2.2 in \cite{p=h}. Since the statement is very useful and the proof is extremely simple, we include in the following.

\begin{lemma}\label{int>0}

If $C$ is an irreducible $J$-holomorphic curve with $C^2\ge 0$ and $SW(e)\ne 0$, then $e \cdot [C]\ge 0$.

\end{lemma}
\begin{proof}

Since $SW(e)\ne 0$, we can represent $e$ by a possible reducible $J$-holomorphic subvariety. Since each irreducible curve $C'$ has $[C']\cdot [C]\ge 0$, we have $e\cdot [C]\ge 0$.
\end{proof}

Let us now fix the notation. 
Since the blowups of $S^2\times \Sigma_h$ and nontrivial $S^2$ bundle over $\Sigma_h$ are diffeomorphic, we will write $M=S^2\times \Sigma_h\#k\overline{\mathbb CP^2}$ if it is not minimal. Let $U$ be the class of $\{pt\}\times \Sigma_h$ which has $U^2=0$ and $T$ be the class of the fiber $S^2\times \{pt\}$. Then the canonical class $K=-2U+(2h-2)T+\sum_iE_i$. 

On the other hand, if $M$ is a nontrivial $S^2$ bundle over $\Sigma_h$, $U$ represents the class of a section with $U^2=1$ and $T$ is the class of the fiber. Then $K=-2U+(2h-1)T$. In this section, we assume $h\ge 1$, {\it i.e.} $M$ is an irrational ruled surface.

We will first show that there is an embedded curve in the fiber class.

\begin{prop}\label{smoothT}
Let $J$ be a tamed almost complex structure on irrational ruled surface $M$, then the fiber class $T$ is $J$-nef. Hence there is an embedded curve in class $T$.  
\end{prop}
\begin{proof}
The first statement is equivalent to the following: let $C$ be an irreducible curve with $[C]=aU+bT-\sum_ic_iE_i$, then $a\ge 0$. We prove it by contradiction. Assume there is an irreducible curve with $a<0$. Then we know that $2g_J([C])-2=C^2+K\cdot [C]$. 
We take the projection $f:C\rightarrow \Sigma_h$ to the base. Its mapping degree is $a=[C]\cdot T$. Since $\Sigma_h$ has genus at least one, by Kneser's theorem, we have $$C^2+K\cdot [C]=2g_J([C])-2\ge 2g(\Sigma_C)-2\ge |a|(2h-2)\ge 0.$$Here $\Sigma_C$ is the model curve of the irreducible subvariety $C$.

Now we look at the class $-[C]$. By the above calculation, we have the Seiberg-Witten dimension $\dim_{SW}(-[C])=C^2-K\cdot (-[C])\ge 0$. Hence, we could apply the wall-crossing formula \begin{equation}\label{wallcross}|SW(K+[C])-SW(-[C])|=|1+(-[C])\cdot T|^h=(1-a)^h\ne 0.\end{equation}
 For classes $T$ and $e=K+[C]=(a-2)U+(2h-2+b)T+\sum_i(1-c_i)E_i$ when $M$ is not a nontrivial $S^2$ bundle  (or $e=K+[C]=(a-2)U+(2h-1+b)T$ when $M$ is a nontrivial $S^2$ bundle), we have $e\cdot T=a-2<0$. We choose an almost complex structure $J'$ such that there is an embedded $J'$-holomorphic curve in class $T$. Then apply Lemma \ref{int>0} for this $J'$ to conclude that  $SW(K+[C])=0$. Apply \eqref{wallcross}, we have $SW(-[C])\ne 0$. Hence the class $0=[C]+(-[C])$ is a class of subvariety. This contradicts to the fact that $J$ is tamed which implies that any positive combinations of curve classes have positive paring with a symplectic form taming $J$. This finishes the proof that $T$ is $J$-nef.
 
Note $g_J(T)=0$, any irreducible curve in class $T$ would be smooth. Hence, we only need to show the existence of an irreducible curve in class $T$. By Theorem 1.5 of \cite{LZrc}, all components of reducible curves in class $T$ are rational curves since $T$ is $J$-nef. Furthermore, all the subvarieties are connected since $J$ is tamed. Then by the dimension counting formula Equation \eqref{red-dim} for reducible subvarieties,  we know $\sum l_{e_i}\le l_T-1=0$. Here $e_i$ is the homology class of each irreducible component and $l_{e_i}=\max\{0, e_i\cdot e_i+1\}$. Hence $l_{e_i}=0$ and all these irreducible components are rational curves of negative self-intersections. 
It is direct to see from the adjunction formula that there are finitely many negative $J$-holomorphic spheres on an irrational ruled surface. For a complete classification of symplectic spheres on irrational ruled surfaces, see \cite{DLW} section $6$. 

Since $SW(T)\ne 0$ and $\dim_{SW}(T)=2$, any point of $M$ lies on a subvariety in class $T$. Since the part covered by reducible curves is a union of finitely many rational curves, as we have shown above, we conclude that there has to be an irreducible, thus embedded, rational curve in class $T$.
\end{proof}
\begin{cor}\label{+sph}
On irrational ruled surfaces, the only irreducible rational curves with nonnegative square are in the fiber class $T$.
\end{cor}
\begin{proof}
Let $[C]=aU+bT-\sum_ic_iE_i$ be the class of an irreducible rational curve. By Proposition \ref{smoothT}, we have $a\ge 0$. Since $g_J([C])=0$, as argued in Proposition \ref{smoothT} by Kneser's theorem, we will have contradiction if $a>0$. Hence we must have $a=0$. Then $C^2=-\sum_ic_i^2\ge 0$. Hence $c_i=0$ for all $i$ and $[C]=bT$. Since $C$ is a rational curve, $-2=C^2+K\cdot [C]=-2b$. Hence $b=1$ and $[C]=T$.
\end{proof}

We can now confirm Question 4.18 of \cite{p=h} for irrational ruled surfaces, and further show there is a unique subvariety in each exceptional class. We rephrase Theorem \ref{intro1}. 

\begin{theorem}\label{red-1}
Let $M$ be an irrational ruled surface, and let $E$ be an exceptional class. Then for any subvariety $\Theta=\{(C_i, m_i)\}$ in class $E$, each irreducible component $C_i$ is a rational curve of negative self-intersection. Moreover, the moduli space $\mathcal M_E$ is a single point.
\end{theorem}
Notice the statement is not true for a rational surface. See \cite{p=h} for a disconnected example and Section 6.1 for a connected example and related discussion.
\begin{proof}
As explained in Corollary \ref{+sph}, any rational curve class must be like $[C]=bT-\sum_ic_iE_i$. If it is the class of an exceptional curve, then $$K\cdot [C]=-2b+\sum c_i=-1, \, \, C^2=-\sum c_i^2=-1.$$
Hence the only such classes are $E_i$ and $T-E_i$. Both types have non-trivial Seiberg-Witten invariants. Hence, there are $J$-holomorphic subvarieties in both types of classes for arbitrary tamed $J$.

Let $\Theta_i\in \mathcal M_{E_i}$ and $\tilde{\Theta}_i\in \mathcal M_{T-E_i}$. Since $T=E_i+(T-E_i)$, we have $\{\Theta_i, \tilde{\Theta}_i\}\in \mathcal M_T$. Since $T$ is $J$-nef by Proposition \ref{smoothT}, we know all irreducible components in $\Theta_i$ and $\tilde{\Theta}_i$ are rational curves by Theorem 1.5 of \cite{LZrc}. Moreover, by Equation \eqref{red-dim}, we have $\sum l_{e_{C_i}}\le l_T-1=0$. Hence $e_{C_i}^2< 0$. This proves the first statement.

For the second statement, we apply the same trick. If there is another subvariety $\Theta_i'\in \mathcal M_{E_i}$. Consider $\Theta=\{\Theta_i, \tilde{\Theta}_i\}\in \mathcal M_T$ and $\Theta'=\{\Theta_i', \tilde{\Theta}_i\}\in \mathcal M_T$. They have common components including $\tilde{\Theta}_i$. We then follow the argument of Lemma \ref{uniquereducible}. After discarding all common components, we have cohomologous subvarieties $\Theta_0$ and $\Theta_0'$. Moreover, we have 
\begin{equation}\label{red-1uni}
0=T^2\ge T\cdot e_{\Theta_0}>e_{\Theta_0}^2=e_{\Theta_0}\cdot e_{\Theta_0'}.
\end{equation}
 The first inequality follows from nefness of $T$. Actually, $T^2=T\cdot e_{\Theta_0}$ by nefness of $T$ applying to the common components we have discarded. The second inequality is because original $\Theta, \Theta'$ have common components at least from $\tilde{\Theta}_i$, and because they are connected  by Theorem 1.5 of \cite{LZrc}.

The inequality \eqref{red-1uni} implies $\Theta_0=\Theta_0'$ by local positivity of intersections and in turn $\Theta=\Theta'$. Hence there is a unique subvariety $\Theta_i$ in each exceptional class $E_i$. Similarly, there is a unique subvariety $\tilde{\Theta}_i$ in $T-E_i$.
\end{proof}

By the uniqueness result that $\mathcal M_E$ is a single point, we know the $J$-holomorphic subvariety in class $E$ is connected and has no cycle in its underlying graph for any tamed $J$ by Gromov compactness. This is because $E$ is represented by a smooth rational curve for a generic tamed almost complex structure, and the above properties hold for the Gromov limit of these smooth pseudoholomorphic rational curves. 

\begin{cor}\label{-pair}
Let $M$ be an irrational ruled surface, and $E$ an exceptional class. If an irreducible $J$-holomorphic curve $C$ satisfies $E\cdot [C]<0$, then $C$ is a rational curve of negative square.
\end{cor}
\begin{proof}
Since $SW(E)\ne 0$, we always have a subvariety in class $E$. By Proposition \ref{red-1}, all irreducible components are negative rational curves. Thus, if $C$ has positive genus, then $C$ cannot be an irreducible component of the $J$-holomorphic subvariety in class $E$. Hence $E\cdot C\ge 0$ by local positivity of intersections. 
\end{proof}

We would like to remark that the technique we use to prove Proposition \ref{smoothT} could also be applied to other situations. Let us summarize it in the following. We will focus on the case when $b^+=1$. To show certain class $A$ with $A^2\ge 0$ is $J$-nef when $J$ is tamed, we would have to show classes $B$ with $A\cdot B<0$ are not curve classes. If such a curve class exists with $B^2\ge 0$ and at the same time $A$ is realized by a symplectic surface, then there is a contradiction due to the light cone lemma. 

Hence we could assume $B^2<0$. For this case, the first obvious obstruction is from the adjunction formula. Second type of obstruction is what we have applied above. To show $B$ is not in the curve cone, we look for classes $C_i$ with nontrivial Seiberg-Witten invariants, and $a_0B+\sum_ia_iC_i=0$ with each $a_i>0$. In Proposition \ref{smoothT}, we choose $a_0=a_1=1$ which are the only nonzero $a_i$'s. For another such application, see Lemma \ref{nonnefS2}. The key observation in this case is $2g_J(B)-2=\dim_{SW}(-B)$. Hence, if $g_J(B)> 0$ and $(K+B)\cdot A<0$ we could efficiently apply the general wall crossing formula in \cite{LLwall, LL} to get nontriviality of Seiberg-Witten invariant for $B$. The above argument could have some obvious twists such as taking $C_1=-kB$. 

For the case of $g_J(B)=0$, we will use a different strategy. We might apply the classifications of negative rational curves, {\it e.g.} \cite{p=h, DLW}, and calculate the intersection numbers with $A$ directly.

Now, we will investigate the moduli space of the subvarieties in class $T$. First, we need a curve to model the moduli space as we did in \cite{LZ-generic}.

\begin{prop}\label{smoothsection}
There is a smooth section of the irrational ruled surface, {\it i.e.} there is an embedded $J$-holomorphic curve $C$ of genus $h$ such that $[C]\cdot T=1$.
\end{prop}
\begin{proof}
We do our calculation for $M=S^2\times \Sigma_h\#k\overline{\mathbb CP^2}$. When $M$ is a nontrivial $S^2$ bundle over $\Sigma_h$, the calculation is similar. 

In Proposition \ref{smoothT}, we have shown that all curves having the homology class $aU+bT-\sum_i c_iE_i$ must have $a\ge 0$. Especially, for a possibly reducible section which is in the class $U+bT-\sum_i c_iE_i$, there is exactly one irreducible component of it has $a=1$ (with multiplicity one), all the others have $a=0$.

Furthermore, let $A=U+hT$, we have $\dim_{SW}(A)=A^2-K\cdot A=2h-(-2h+2h-2)> 0$. Since $K-A=-3U+(h-2)T+\sum_iE_i$ pairs negatively with $T$, by Lemma \ref{int>0}, $SW(K-A)= 0$. Apply the wall crossing formula, we have $SW(A)=\pm 2^h\ne 0$. Hence there is a subvariety in class $U+hT$. Choose an irreducible component with $a=1$, call it $C$.

We show that $C$ has to be smooth. Since $[C]\cdot T=1$, for any point $x\in C$, there is a subvariety $\Theta_x$ in class $T$ passing through it. Since any curve class $aU+bT-\sum_i c_iE_i$ has $a\ge 0$, we know $C$ cannot be an irreducible component of this subvariety $\Theta_x$ in class $T$. If $x$ is a singular point, the contribution to the intersection of $C$ and $\Theta_x$ would be greater than $1$. Hence by the local positivity of the intersection, we know $C$ is an embedded curve.

Since $C$ is a section, we have $g(C)>0$ by Kneser's theorem. By Corollary \ref{-pair}, for any exceptional rational curve class $E$, we have $[C]\cdot E\ge 0$. Since $T-E$ is another exceptional rational curve class and $[C]\cdot (T-E)+[C]\cdot E=[C]\cdot T=1$, we have $0\le [C]\cdot E\le 1$. Because of this, $$K\cdot [C]+[C]^2=(2h-2-2b+\sum c_i)+(2b-\sum c_i^2)=2h-2.$$ Hence $C$ has genus $h$.
\end{proof}

We are ready to show the structure of the moduli space $\mathcal M_T$. 

\begin{theorem}\label{M_T}
Let $M$ be an irrational ruled surface of base genus $h$. Then for any tamed $J$ on $M$, 
\begin{enumerate}
\item there is a unique subvariety in class $T$ passing through a given point;
\item the moduli space $\mathcal M_T$ of the subvarieties in class $T$ is homeomorphic to $\Sigma_h$; 
\item $\mathcal M_{red, T}$ is a set of finitely many points.
\end{enumerate}
\end{theorem}
\begin{proof}
Let $C\cong \Sigma_h$ be the smooth $J$-holomorphic section constructed in Proposition \ref{smoothsection}. First, by Lemma \ref{uniquereducible}, for any given point $x\in M$, there is a unique element in $\mathcal M_T$ passing through it. We denote this element by $\Theta_x$.

Now, we construct a natural map $h: x\mapsto \Theta_x$ from  $C$ to $\mathcal M_T$. The map $h$ is surjective because $T\cdot [C]\ne 0$. The map is injective since $T\cdot [C]=1$ and the positivity of intersection. To show $h^{-1}$ is continuous,  consider a sequence $\Theta_i\in \mathcal M_T$ approaching to its Gromov-Hausdorff limit $\Theta$. Let the intersection points of $\Theta_i, \Theta$ with $C$ be $p_i, p$. Then $p_i$ has to approach $p$ by the first item of the definition of topology on $\mathcal M_T$. Now since $C\cong \Sigma_h$ is Hausdorff and $\mathcal M_T$ is compact, the fact we just proved that $h^{-1}: \mathcal M_T\rightarrow C$ is continuous would imply $h$ is also continuous. Hence $h$ is a homeomorphism. This completes the proof of the second statement.

The third bullet, that $\mathcal M_{red, T}$ is a set of finitely many points, follows from the following two facts. First, each irreducible component of an element in $\mathcal M_{red, T}$ would have negative self-intersection since $\sum l_{e_i}\le 0$ by Equation \eqref{red-dim}. Second, there are finitely many negative rational curves as we have seen in Proposition \ref{smoothT}.
\end{proof}

\begin{cor}\label{spherefiber}
Every irreducible rational curve belongs to a fiber, {\it i.e.} it is an irreducible component of an element of $\mathcal M_T$.
\end{cor}
\begin{proof}
First, by Corollary \ref{+sph}, all irreducible rational curves with nonnegative square have class $T$. Hence, we could only talk about negative curves. By Kneser theorem, for such a curve $C$, we have $[C]\cdot T=0$ as argued in Corollary \ref{+sph}. By Theorem \ref{M_T} (1), for any point $x\in C$, there is a unique element $\Theta_x\in \mathcal M_T$ passing through it. If $C$ is not an irreducible component of $\Theta_x$, then $[C]\cdot T>0$ by the positivity of intersection, which contradicts to $[C]\cdot T=0$.
\end{proof}

Theorem \ref{M_T} and Corollary \ref{spherefiber} constitute Theorem \ref{intro2} in the introduction.

Along with Corollary \ref{+sph}, we have described  $\mathcal M_e$ for any rational curve class $e$ and an arbitrary tamed almost complex structure on an irrational ruled surface.

Some finer local structures of the moduli space are described in the following.

\begin{cor}\label{fibration}
The natural map $f: M\rightarrow \mathcal M_T$, where $f(x)$ is the unique subvariety $\Theta_x$ in class $T$ passing through $x$, is a continuous map.
\end{cor}
\begin{proof}
We only need to show that for any sequence $\{x_n\}_{n=1}^{\infty}$ converging to $x$, the subvarieties $\Theta_{x_n}$ converge to $\Theta_x$ in $\mathcal M_T$. We notice that if a sequence satisfies $\lim_{n\rightarrow \infty}\rho(\Theta_x, \Theta_{x_n})=0$ (the first defining property of the topology of $\mathcal M_e$), then a subsequence must converge to $\Theta_x$ because of Theorem \ref{M_T}(1).

Hence, we can assume on the contrary that there is a sequence $\{x_n\}_{n=1}^{\infty}$ converging to $x$ such that $\rho(\Theta_x, \Theta_{x_n})>c>0$ for a constant $c$. However, since $\mathcal M_T$ is compact, we know there is a subsequence of $\{x_n\}$ such that it converges to a subvariety $\Theta'\in \mathcal M_T$. Since $\{x_n\}$ converging to $x$, we know $x\in |\Theta'|\cap |\Theta_x|$, which implies $\Theta'=\Theta_x$ by Theorem \ref{M_T} (1). This contradicts to our assumption $\rho(\Theta_x, \Theta_{x_n})>c>0$. Thus we know $f: M\rightarrow \mathcal M_T$ is a continuous map.
\end{proof}

It is worth pointing out that near a smooth curve $C\subset \mathcal M_T$ (or more generally any moduli space $\mathcal M_e$), the convergence is very explicit, as described in \cite{T1} (see also \cite{LZ-generic}). Recall that any curve in a neighborhood of $C$ in $\mathcal M_T$ can be written as $\exp_C(\eta)$ with $\eta$ being a section of normal bundle $N_C$ satisfying 
\begin{equation}\label{localdef}
D_C\eta+\tau_1\partial \eta+\tau_0=0.
\end{equation} 
 Here, $\tau_1$ and $\tau_0$ are smooth, fiber preserving maps from a small radius disk in $N_C$ to $\hbox{Hom}(N_C\otimes T^{1,0}C; N_C\otimes T^{0,1}C)$ and to $N_C\otimes T^{0,1}C$ that obey $|\tau_1(b)|\le c_0|b|$ and $|\tau_0(b)|\le c_0|b|^2$. Meanwhile, $D_C$ is the $\mathbb R$-linear operator that appears in (2.12) of \cite{T1}, which is used to describe the first order deformations of $C$ as a $J$-holomorphic submanifold. The $L^2$-orthogonal projection map from $C^{\infty}(C; N_C)$ to the kernel of $D_C$ maps an open set of solutions of \eqref{localdef} diffeomorphically to an open ball centered at $0$ in $\ker(D_C)$. Notice in our situation, $\ker(D_C)$ has complex dimension one. This description identifies an open neighborhood $\mathcal N(C)$ of $C$ in $\mathcal M_T$ with a small radius ball about the origin in $\ker(D_C)$. From this description, we know the tangent bundle of each element in $\mathcal N(C)$ varies as a smooth family.

We will finish this section by a digression on another example of using the technique of Proposition \ref{smoothT}, now for rational surfaces $M=\mathbb CP^2\# k\overline{\mathbb CP^2}$.

\begin{lemma}\label{nonnefS2}
Let $M$ be a rational surface and $J$ be tamed. Let  $A\in H^2(M, \mathbb Z)$ be a class with $A^2\ge 0$. Moreover, assume there is an embedded $J'$-holomorphic curve in class $A$ for a tamed $J'$. 
Then if a $J$-holomorphic curve $C$ such that $[C]\cdot A<\min\{0, -K\cdot A\}$, it has to be a rational curve with negative square.
\end{lemma}
For example, $A$ could be chosen as $H$, $H-E$, $3H-E$, {\it etc}.
\begin{proof}

We first show $C$ is a rational curve by contradiction. If $g_J([C])>0$, we have $C^2+K\cdot [C]\ge 0$. We look at the class $-[C]$, which has $\dim_{SW}(-[C])=C^2+K\cdot [C]\ge 0$. The wall-crossing formula for rational surfaces implies $|SW(K+[C])-SW(-[C])|=1$. For classes $A$ and $e=K+[C]$, we have $A\cdot e<0$. Apply Lemma \ref{int>0}, using the conditions $A^2\ge 0$ and $A$ has an embedded $J'$-holomorphic representative, we conclude $SW(K+[C])=0$. Hence $SW(-[C])\ne 0$ by wall-crossing. It follows that the class $0=[C]+(-[C])$ is a class of subvariety, which contradicts to the fact that $J$ is tamed. Hence $C$ has to be a rational curve.

Now we choose an integral symplectic form $\omega$ taming $J$. Hence, for large $N$ we have $\dim_{SW}(N[\omega])>0$. Moreover, the class $K-N[\omega]$ pairs negatively with the symplectic form $\omega$ for large $N$. Therefore, we must have $SW(K-N[\omega])=0$. By wall-crossing, we have $SW(N[\omega])\ne 0$ for large $N$. Then by Lemma \ref{int>0}, we have $[\omega]\cdot A\ge 0$. Since $C$ is a $J$-holomorphic curve, $[\omega]\cdot [C]>0$. If $C^2\ge 0$, and because $A^2\ge 0$, we apply the light cone lemma to conclude that $[C]\cdot A\ge 0$, which contradicts to our assumption. Hence $C$ is a rational curve with negative square.
\end{proof}

With this lemma in hand, we could apply the classification of negative rational curves in \cite{p=h} to find $J$-nef classes  for rational surfaces with $K^2> 0$.

The only feature of rational surfaces used in the proof is that they have nonzero wall-crossing number for all the classes with non-negative Seiberg-Witten dimension. Hence, the argument could be extended to a general symplectic $4$-manifold with $b^+=1$ under this assumption.

\section{Rational surfaces}
In this section, we will concentrate on rational surfaces, {\it i.e.} $4$-manifolds diffeomorphic to $\mathbb CP^2\#k\overline{\mathbb CP^2}$ or $S^2\times S^2$.  We study the moduli space of $J$-holomorphic subvarieties in a sphere class. Our main results, Theorem \ref{homeoCPl} and Proposition \ref{connected}, show that our moduli space behaves like a linear system in algebraic geometry. 
\subsection{Connectedness of moduli spaces of subvarieties}

For the applications, in particular the symplectic isotopy problem, it is important to show that the reducible part $\mathcal M_{red}$ would not disconnect the whole moduli space. This is the technical heart of \cite{ST} in which it is called Isotopy Lemma. In our setting, we would show the following stronger result.

\begin{prop}\label{connected}
Suppose $e$ is a $J$-nef class with $g_J(e)=0$. Then $\mathcal M_e$ and $\mathcal M_{irr, e}$ are path connected. In particular, any two smooth rational curves representing class $e$ are connected by a path of smooth rational curves. 
\end{prop}

\begin{proof}
We divide our argument into five parts.

{\bf Part 1: Reduce to rational surfaces.}

If $\mathcal M_e\ne \emptyset$, since $e$ is $J$-nef, we know the self-intersection $e^2\ge 0$. By a classical result of McDuff, if furthermore $g_J(e)=0$ and $\mathcal M_{irr, e}\ne \emptyset$, then $M$ has to be a rational or ruled surface. When $M$ is not rational, the results follow from Theorem \ref{M_T} and Corollary \ref{spherefiber}. Hence, in the following, we assume $M$ is a rational surface. 

{\bf Part 2: Definition of pretty generic tuples.}

We first assume $e$ is a big $J$-nef class, {\it i.e.} a $J$-nef class with $e\cdot e>0$. For the proof we need the following definition of \cite{LZ-generic}. We denote $l=l_e\ge 2$. Let $M^{[k]}$ be the set of $k$ tuples of pairwise distinct points in $M$.

\begin{definition}\label{pretty generic}
Fix a point $x\in M(e)$ (see Section \ref{secnef} for the definition). An element  $\Omega\in M^{[l-2]}$ is called {\it pretty generic} with respect to $e$ and $x$ if
\begin{itemize}
\item $x$ is distinct from any entry of $\Omega$;
\end{itemize}
For each $\Theta=\{(C_1, m_1),\cdots, (C_n, m_n)\}\in \mathcal M^x_{red, e}$ with $x\in C_1$,
\begin{itemize} 

\item $x$ is not in $C_i$ for any $i\geq 2$; 

\item  $\Omega_i\cap \Omega_j=\emptyset$ for $i\ne j$, where $\Omega_i=\Omega\cap C_i$;

\item  
$1+w_1=m_1e\cdot e_1(\ge l_{e_1})$, and  $w_i=m_ie\cdot e_i(\ge l_{e_i})$ for $i\geq 2$. Here $w_i$ is  the cardinality of $\Omega_i$. 
\end{itemize}
Let $G_{e}^x$ be the set of pretty generic $l-2$ tuples with respect to $e$ and $x$. 
\end{definition}

It is indeed a generic set in the sense that the complement of $G_e^x$ has complex codimension at least one in $M^{[l-2]}$ by Proposition 4.8 of \cite{LZ-generic}. In particular, the set $G_e^x$ is path connected.

{\bf Part 3: $\mathcal M_{irr, e}$ is path connected when $e$ is a big $J$-nef class.}

Now, if $C$ and $C'$ are smooth rational curves in $\mathcal M_{irr, e}$, they intersect at $l-1$ points (counted with multiplicities). If one of the intersection points $\tilde x\in D$ where $D$ is an irreducible curve in $Z(e)$, we have $e_C\cdot e_D=0$ by definition. On the other hand, since $C$ is irreducible, we know the irreducible curve $D$ is not identical to $C$ since the class $e=e_C$ is big and thus $e \cdot e_C>0=e\cdot e_D$. Then $\tilde x\in C\cap D$ implies $e\cdot e_D>0$ which is a contradiction. 

Hence, all of these intersection points are in $M(e)$. First, we will show that we can deform the curve $C$ within $\mathcal M_{irr, e}$ to a smooth curve $\tilde{C}$ such that all the intersection points with $C'$ are of multiplicity one, if there are intersection points of $C$ and $C'$ having multiplicity greater than one. We know $\mathcal M_{irr, e}$ is a smooth manifold of dimension $2l$. Hence we can choose an open neighborhood $U$ of $C\in \mathcal M_{irr, e}$. We look at the intersection points between elements in $U$ and the curve $C'$. There are $l-1$ intersection points counted with multiplicities. Let $U'\subset U$ be a subset of $U$ such that an element in $U'$ is tangent to the curve $C'$, {\it i.e.} intersecting at least one point with multiplicity at least two. In particular, $C\in U'$. 

The following is a general fact of automatic transversality, see {\it e.g.} Remark 3.6 of \cite{LZ-generic}. If we have $k\le l$ distinct points $x_1, \cdots, x_k$ in $C'$ and $k'<k$ with $k+k'\le l$, then the set of smooth rational curves in class $e$ passing through $x_1, \cdots, x_k$ and having the same tangent space at the $k'$ points $x_1, \cdots, x_{k'}$ as $C'$ is still a smooth manifold, whose dimension is $2(l-k-k')$. Since we can vary $x_1, \cdots, x_k$ in the curve $C'$, and $k'\ge 1$, we know $U'$ is a submanifold of U with dimension $2(l-k-1)+2k=2l-2$. In particular, $U\setminus U'$, which is the set of curves in $U$ intersecting $C'$ at points with multiplicity one, is non-empty and path connected. Moreover, elements in $U\setminus U'$ could be connected by paths to the element $C\in U'$ within $U\setminus U'$, in the sense that for any $\tilde C \in U\setminus U'$ there is a path $P(t)\subset U$ such that $P(1)=C$, $P(0)=\tilde C$ and $P([0, 1))\subset U\setminus U'$. Hence, any curve $\tilde C\in \mathcal M_{irr, e}$ could be obtained by deforming the curve $C$ within $\mathcal M_{irr, e}$, such that all the intersection points of $\tilde C$ and $C'$ are of multiplicity one. For simplicity of notation, we will still write the deformed curve $\tilde C$ by $C$.

We can now choose one of the intersection points of $C$ and $C'$, and call it $x$. For the remaining $l-2$ points $x_3, \cdots, x_l$, they might not be in $G_e^x$. 
 Choose two more points $y\in C$ and $y'\in C'$ other than all these intersection points. We are able to choose $l-2$ disjoint open neighborhoods $N_i$ of $x_i$ in $M$ with $i=3, \cdots, l$, such that all curves representing $e$, passing through $x$, and $y$ or $y'$, and intersecting all $N_i$ are smooth rational curves. This is because $\mathcal M_{irr, e}$ is a smooth manifold by Theorem \ref{unobstructed} and there is a unique subvariety, smooth or not, passing through $l$ given points on an irreducible curve in class $e$.

Since the complement of $G_e^x$ has complex codimension at least one in $M^{[l-2]}$, we are able to choose a pretty generic $l-2$-tuple from $N_3\times \cdots\times N_l$. With these understood, we are able to deform $C$ and $C'$ within $\mathcal M_{irr, e}$ to two smooth rational curves intersecting at $x$ and an $l-2$-tuple in $G_e^x$. We still denote these two curves by $C$ and $C'$.

By Proposition 4.9 of \cite{LZ-generic}, the subset $\mathcal M_{e}^{x, x_3, \cdots, x_l}\subset \mathcal M_e$ is homeomorphic to $\mathbb CP^1=S^2$ when $(x_3, \cdots, x_l)\in G_e^x$. Moreover, $\mathcal M_{e}^{x, x_3, \cdots, x_l}\cap \mathcal M_{red, e}$ is a finite set of points. Since $C, C'\in \mathcal M_{e}^{x, x_3, \cdots, x_l}$, they are connected by a family of smooth rational curve in $\mathcal M_{e}^{x, x_3, \cdots, x_l}\cap \mathcal M_{irr, e}$. This finishes the proof that $\mathcal M_{irr, e}$ is path connected when $e$ is big $J$-nef.

{\bf Part 4: $\mathcal M_e$ is path connected when $e$ is a big $J$-nef class.}

To show $\mathcal M_e$ is path connected, we only need to prove that any element in $\mathcal M_{red, e}$ is connected by a path to an element in $\mathcal M_{irr, e}$. This would imply $\mathcal M_e$ is path connected since we have shown $\mathcal M_{irr, e}$ is path connected. Let $\Theta \in \mathcal M_{red, e}$, we choose $l'=\sum e\cdot e_{C_i}$ distinct points $x_1, \cdots, x_{l'}$ from the smooth part of $\Theta=\{(C_i, m_i)\}$. We choose the $l'$ points such that there are exactly $e\cdot e_{C_i}$ points on $C_i$ for each $i$, each with type $(x_i, m_{i})$ in the sense of section \ref{intersection}. Counted with weights, there are $\sum m_ie\cdot e_{C_i}=l-1$ points. We then choose another point, labeled by $x_l$, from the smooth part of $\Theta$ and different from $x_1, \cdots, x_{l'}$. By Lemma \ref{uniquereducible}, there is a unique element in $\mathcal M_e$ passing through points $x_1, \cdots, x_{l'}$ with multiplicities and another point $x_l$.

We take $l$ disjoint open sets $N_1, \cdots, N_l\subset M$ as following. For each $x_i, i\le l'$, assume it is on the irreducible component $(C_j, m_j)$. We choose $m_j$ disjoint open sets, say $N'_1, \cdots, N'_{m_j}$, such that $\overline{N'_{1}}\cup \cdots \cup \overline{N'_{m_j}}$ is a neighborhood of $x_i$ in $M$ and $x_i\in  \overline{N'_{k}}$ for each $1\le k\le m_i$. Considering all the points $x_1, \cdots, x_{l'}$, there are $l-1=\sum m_ie\cdot e_{C_i}$ such open sets. We relabel them by  $N_1, \cdots, N_{l-1}$.
Finally, we take a neighborhood $N_l$ of $x_l$ in $M$. Apparently, we can choose these open sets such that they are disjoint from each others.

We denote by $\mathcal M_{irr, e, k}$ (resp. $\mathcal M_{red, e, k}$) the subset of $\mathcal M_{irr, e}\times M^{[k]}$ (resp. $\mathcal M_{red, e}\times M^{[k]}$) that consists of elements of the form $(C, x_1, \cdots, x_k)$ with $x_i\in C$ and distinct. There are natural projections $\pi_{irr, l}: \mathcal M_{irr, e, l}\rightarrow M^{[l]}$ and $\pi_{red, l}: \mathcal M_{red, e, l}\rightarrow M^{[l]}$. First, we notice that the diagonal elements $Z_{diag}=M^l\setminus M^{[l]}$ is a finite union of submanifolds of dimension at least two. Proposition 4.5 in \cite{LZ-generic} shows that the image of $\pi_{red, l}$, say $Z_{red}\subset M^{[l]}$, is a finite union of submanifolds of codimension at least two, and $\pi_{irr, e}$ maps onto its complement. Moreover, the map $\pi_{irr, l}$ is one-to-one. Hence, $M^l\setminus (Z_{diag}\cup Z_{red})$ is path connected. In particular, we can choose a path $P(t)$ in $M^l$ such that $P(0)\in M^l$ is the $l$ points with weight $(x_1, m_{k_1}), \cdots, (x_{l'}, m_{k_{l'}})$ and $x_l$ that determine $\Theta$ uniquely and $P((0, 1])\subset N_1\times \cdots \times N_l\setminus Z_{red}$. Since all the $l$ tuples $P(t)$ determines the subvariety uniquely, the path $P(t)\subset M^l$ gives rise to a path connecting $\Theta$ to $\mathcal M_{irr, e}$.

{\bf Part 5: $\mathcal M_e$ is homeomorphic to $S^2$ when $e\cdot e=0$.}

When $e\cdot e=0$, we no longer need the technicalities of pretty generic tuples. In fact, the argument here is similar to that of Theorem \ref{M_T}. Instead of finding a smooth section as in Proposition \ref{smoothsection}, we will use a general construction in \cite{LZ-generic} of a ``dual" $J$-nef class. This will be used as our model for moduli space.

  By Theorem \ref{unobstructed}, $\mathcal M_{irr, e}$ is a manifold of complex dimension $1$ and $\mathcal M_{red, e}$ is a union of finitely many points. We will show that $\mathcal M_e=\mathcal M_{irr, e}\cup \mathcal M_{red, e}$ is actually homeomorphic to $S^2$. By Proposition 4.6 of \cite{LZ-generic}, there is another $J$-nef class $H_e$ with $g_J(H_e)=0$ such that $H_e\cdot e=1$. We choose a smooth rational curve $S$ representing $H_e$. Given any $z\in S$, there is a unique (although possibly reducible) rational curve $C_z$ in class $e$ passing through $z$. Thus we obtain a map $h: z\mapsto C_z$ from $S$ to $\mathcal M_e$. 

The map $h$ is  surjective since $H_e\cdot e\ne 0$. Since $S$ is also $J$-holomorphic and $H_e\cdot e=1$
any curve in $\mathcal M_e$ intersects with $S$ at a unique point by the positivity of intersection. 
Therefore $h$  is also one-to-one.  

Now let us show that $h$ is a homeomorphism, namely both $h$ and $h^{-1}$ are continuous. 
Since $S=S^2$ is Hausdorff  and $\mathcal M_e$ is compact, 
if we can  show that $h^{-1}:\mathcal M_e\to S$ is continuous, it follows that $h$ is also continuous.
 To show $h^{-1}$  is continuous, consider a sequence $C_i \in \mathcal M_e$ approaching to its Gromov-Hausdorff limit $C$. Let the intersection of $C_i$ (resp. $C$) with $S$ be $p_i$ (resp. $p$). Then $p_i$ has to approach $p$ by the first item of the definition of topology on $\mathcal M_e$.
 Therefore
$h$ is a homeomorphism.  
\end{proof}

\subsection{$\mathcal M_e=\mathbb CP^l$ when $e$ is primitive} 
In the argument of Proposition \ref{connected}, we have shown $\mathcal M_e=\mathbb CP^1$ when  $e\cdot e=0$.
We will next generalize it to show that $\mathcal M_{e}$ is homeomorphic to $\mathbb CP^l$ when $e$ is primitive, hence confirm Question 5.25 of \cite{LZ-generic} in the topological sense in this circumstance. This is Theorem \ref{cpl} and we state it again below as Theorem \ref{homeoCPl}.

We first need a lemma to adapt the discussion of section 4.3 in \cite{LZ-generic}. This lemma is crucial in our construction of the model for the moduli space. 

\begin{lemma}\label{H_enef}
Let $M=S^2\times S^2$ or $\mathbb CP^2\#k\overline{\mathbb CP^2}$ with $k\ge 1$. Suppose  $e\in H^2(M, \mathbb Z)$ is a primitive (i.e. $e$ is not divisible by an integer $k>1$) $J$-nef class with $g_J(e)=0$. Then there is a $J$-nef class $H_e$ such that $g_J(H_e)=0$ and $H_e\cdot e=1$. Moreover, $H_e$ can be assumed to be not  proportional to $e$.
\end{lemma}
\begin{proof}

We take the class $H_e$ to be the same ones chosen in the proof of Lemma 4.13 of \cite{LZ-generic} except if $e$ is Cremona equivalent to $H$ or $2H-E_1-E_2$ when $M=\mathbb CP^2\#k\overline{\mathbb CP^2}$. 

When $e$ is equivalent to $H$, without loss of generality, we assume $e=H$. We will show that at least one of $H-E_1, \cdots, H-E_k$ is $J$-nef. Let us first take $H_e'=H-E_1$, and assume there is a curve pairing negatively with it. By Lemma 4.15 of \cite{LZ-generic}, we know an irreducible curve class $e_C=aH-b_1E_1-\cdots -b_kE_k$, pairing negatively with $H-E_1$, must have $a\le 0$. Hence $a=0$, otherwise it contradicts to the assumption that $e=H$ is $J$-nef. But when $a=0$, we have $b_1=-H_e'\cdot e_C>0$. Moreover, since $SW(E_i)\ne 0$, we know there are $J$-holomorphic subvarieties in classes $E_i$. At least one $b_i<0$, otherwise $0$ is a linear combination of $e_C$ and $e_{E_i}$ which contradicts to the fact that $J$ is tamed. Then we look at the adjunction number $$e_C\cdot e_C+K_J\cdot e_C=-b_1^2-\cdots-b_k^2+b_1+\cdots +b_k\le 0.$$ To make sure the adjunction number is no less than $-2$, we will exactly have one negative $b_i$, say $b_2=-1$. Other $b_i$'s are $0$ or $1$. In particular, $b_1=1$.

Then we take the class $H-E_2$. If it is not $J$-nef, we can argue as in the last paragraph for $H-E_1$ to show that there is a curve class $e_{C_2}=-b_1^{(2)}E_1-\cdots -b_k^{(2)}E_k$ with only one negative coefficient which is $-1$, and others are $0$ or $1$. If the negative coefficient is some $b_i^{(2)}=-1$ such that $b_i=1$, then $e_C+e_{C_2}$ is a linear combination of $E_1, \cdots, E_k$ with non-positive coefficients. This contradicts to the fact that $J$ is tamed. If the negative coefficient is some $b_i^{(2)}$ ($i\ne 2$ since $b_2^{(2)}=1$) with  $b_i=0$, say $b_3^{(2)}=-1$, we could continue our argument with the class $H-E_3$. Since $k$ is a finite number, this process will stop at some finite number $j$, such that when we argue it with a non $J$-nef class $H-E_{j}$, we will get a curve class $e_{C_{j}}$ with one negative $b_i^{(j)}$ and $i<j$. Then the sum $e_{C_i}+\cdots +e_{C_j}$ is a linear combination of $E_i$ with non-positive coefficients, which contradicts to the tameness of $J$ again. Hence, we have shown that at least one of $H-E_1, \cdots, H-E_k$ is $J$-nef. We choose it as $H_e$, which is a class satisfying the requirements of our lemma.

When $e$ is equivalent to $2H-E_1-E_2$, say $e=2H-E_1-E_2$, we claim that one of the classes, $H-E_1$ or $H-E_2$, is $J$-nef. We assume both $H-E_1$ and $H-E_2$ are not $J$-nef. By Lemma 4.15 of \cite{LZ-generic}, we know an irreducible curve class $e_C=aH-b_1E_1-\cdots -b_kE_k$ pairing negatively with $H-E_1$ or $H-E_2$ must have $a\le 0$. We are able to determine all the possible classes that pair negatively with $H-E_1$. In this case, $(H-E_1)\cdot e_C\le -1$ implies $a\le b_1-1$. Since $2H-E_1-E_2$ is $J$-nef, we know $H-E_2$ pairs positively with $e_C$, which implies $b_2\le a-1\le -1$. We calculate the $K_J$-adjunction number
$$e_C\cdot e_C+K_J\cdot e_C\le a^2-b_2^2-3a+b_2\le 2a-1-3a+b_2\le -2.$$

The equality holds only when $b_2=a-1$ (the second inequality) and all other $b_i$ are $0$ or $1$ (the first inequality). Furthermore $b_2=a-1$ would imply $a=b_1-1$ also holds. Hence the only possible classes are $E_2-E_1-b_3E_3-\cdots -b_kE_k$ and $-H+2E_2-b_3E_3-\cdots -b_kE_k$ with all $b_3, \cdots b_k$ are $0$ or $1$. 

Similarly, if the class $H-E_2$ is not $J$-nef, then there is a curve class $E_1-E_2-b_3E_3-\cdots -b_kE_k$ or $-H+2E_1-b_3E_3-\cdots -b_kE_k$ with all $b_3, \cdots b_k$ are $0$ or $1$. 

We notice the classes $E_2-E_1-b_3E_3-\cdots -b_kE_k$ and $E_1-E_2-b'_3E_3-\cdots -b'_kE_k$ cannot coexist. We assume there is no curve class of type $E_2-E_1-b_3E_3-\cdots -b_kE_k$. Then there is a curve in class $-H+2E_2-b_3E_3-\cdots -b_kE_k$. At the same time, there is a curve in class  $E_1-E_2-b'_3E_3-\cdots -b'_kE_k$ or $-H+2E_1-b'_3E_3-\cdots -b'_kE_k$. In particular, it implies that there are $J$-holomorphic subvarieties in classes $-H+2E_2-b_3E_3-\cdots -b_kE_k$ and $-H+2E_1-b'_3E_3-\cdots -b'_kE_k$. Again, $SW(E_i)\ne 0$ implies that there are $J$-holomorphic subvarieties in classes $E_i$. In turn, it would imply that there are subvarieties in classes $-H+2E_1$ and $-H+2E_2$. Finally $SW(H-E_1-E_2)\ne 0$, hence $H-E_1-E_2$ is the class of a subvariety. However, then we know $0=(-H+2E_1)+(-H+2E_2)+2(H-E_1-E_2)$ is the class of a subvariety, which contradicts to the tameness of $J$.

Hence, $H-E_1$ or $H-E_2$ has to be $J$-nef. It is our $H_e$ when $e=2H-E_1-E_2$. It satisfies all the requirements. This finishes the proof of the lemma. 
\end{proof}

\begin{theorem}\label{homeoCPl}
Suppose $J$ is a tamed almost complex structure on a rational surface $M$, and $e$ is a primitive class which is represented by a smooth $J$-holomorphic sphere. Then $\mathcal M_e$ is homeomorphic to $\mathbb CP^l$ where $l=\max\{0, e\cdot e+1\}$.
\end{theorem}
\begin{proof}
When $e\cdot e<0$, we have $l=0$.  It follows from positivity of local intersections that the smooth $J$-holomorphic sphere representing $e$ is the unique element in $\mathcal M_e$.

When $e\cdot e\ge 0$, we know $e$ is $J$-nef. 
We could assume $e\cdot e>0$ otherwise it is verified in Proposition \ref{connected}. For $M=\mathbb CP^2$, it is well known $\mathcal M_H=\mathbb CP^2$, see {\it e.g.} \cite{Gr}. Actually, $\mathcal M_{2H}=\mathbb CP^5$ by the result of \cite{Sik}. Hence, let us discuss $\mathbb CP^2\#k\overline{\mathbb CP^2}$
or $S^2\times S^2$. Our first goal is to find a class $e'$ such that it is $J$-nef, $g_J(e')=0$ and $e'\cdot e=l$. 

In Lemma \ref{H_enef}, we have found $J$-nef class $H_e$  such that $g_J(H_e)=0$ and $H_e\cdot e=1$ when $e$ is primitive. Moreover, we could choose $H_e$ not proportional to $e$. 

Let $e'=e+H_e$. By adjunction formula $$(e+H_e)^2+K\cdot (e+H_e)=(-2)+(-2)+2=-2,$$
 $g_J(e')=0$ and $e'$ is $J$-nef since both $e$ and $H_e$ are so. Moreover, the intersection number $e'\cdot e= e^2+1=l$. Since $e'^2>0$, we have $\dim_{SW}(e')>0$. If $SW(K-e')\ne 0$, it will contradict to the nefness of $e'$ by $e\cdot (K-e')=-\dim_{SW}(e')<0$. Hence by Seiberg-Witten wall-crossing, we have $SW(e')=1$. By Proposition 4.5 of \cite{LZ-generic}, we choose a smooth rational curve $S$ in class $e'$. Notice by our choice of class $e'$, the smooth rational curve $S$ cannot be an irreducible component of any element in $\mathcal M_e$. 
Since $J$ is tamed, any subvariety in $\mathcal M_e$ is connected. Take points $x_1, \cdots, x_l\in S$. Some of the points $x_i$ might be identical. Since the curve $S$ is given {\it a priori}, when we talk about the intersection of subvarieties as in Section \ref{intersection}, we could also include the second type where the ``matching" curve at $x_i$ is given by $S$.

We will show that there is a unique (possibly reducible) rational curve in class $e$ passing through all $x_i$. The argument is similar to that of Lemma \ref{uniquereducible}, with slight modifications with regard to the existence of the curve $S$ and the corresponding second type intersections. We assume there are two such subvarieties, say $\Theta=\{(C_i, m_i)\}, \Theta'=\{(C_i', m_i')\}$.
If $\Theta, \Theta'$ have no common components, then the result follows from positivity of local intersection since $e_{\Theta}\cdot e_{\Theta'}<l$. 

Hence we assume they have at least one common components. 
In particular, none of $\Theta$ and $\Theta'$ is a smooth variety. 

We rewrite two subvarieties $\Theta, \Theta' \in \mathcal M_e$, allowing $m_i=0$ in the notation, such that they have the same set of irreducible components formally, i.e. $\Theta=\{(C_i, m_i)\}$ and $\Theta'=\{(C_i, m'_i)\}$. Then for each $C_i$, if $m_i\le m'_i$, we change the components to $(C_i, 0)$ and $(C_i, m'_i-m_i)$. Similar procedure applies to the case when $m_i>m_i'$. Apply this process to all $i$ and discard finally all components with multiplicity $0$ and denote them by $\Theta_0,\Theta'_0$ and still use $(C_i, m_i)$ and $(C_i, m'_i)$ to denote their components. Notice they are homologous, formally have homology class $$e-\sum_{m_{k_i}<m_{k_i}'} m_{k_i}e_{C_{k_i}}-\sum_{m_{l_j}'<m_{l_j}} m'_{l_j}e_{C_{l_j}}-\sum_{m_{q_p}'=m_{q_p}} m'_{q_p}e_{C_{q_p}}.$$

There are two ways to express the class, by taking $e=e_{\Theta}$ or $e=e_{\Theta'}$ in the above formula. Namely, it is $$\sum_{m_{k_i}<m_{k_i}'}(m_{k_i}'-m_{k_i})e_{C_{k_i}}+\hbox{others}=e_{\Theta_0'}=e_{\Theta_0}=\sum_{m_{l_j}'<m_{l_j}} (m_{l_j}-m'_{l_j})e_{C_{l_j}}+\hbox{others}.$$
Here the term ``others" means the terms $m_ie_{C_i}$ or $m_i'e_{C_i}$ where $i$ is not taken from $k_i$, $l_j$ or $q_p$. 

Now $\Theta_0$ and $\Theta_0'$ have no common components. They intersect the rational curve $S$ at least $e'\cdot e_{\Theta_0}\ge e\cdot e_{\Theta_0}$ points (as a subset of $\{x_1, \cdots, x_l\}$) with multiplicities. In the inequality above, we make use of the fact that $H_e$ is $J$-nef. Hence $\Theta_0$ and $\Theta_0'$ would intersect at least $e\cdot e_{\Theta_0}$ points with multiplicities.

We notice that $e\cdot e_{\Theta_0}\ge e_{\Theta_0}\cdot e_{\Theta_0}$. In fact, the difference $e-e_{\Theta_0}=e-e_{\Theta_0'}$ has $3$ types of terms, any of them pairing non-negatively with $e_{\Theta_0}=e_{\Theta_0'}$. For the terms with index $k_i$, {\it i.e.} the terms with $m_{k_i}<m_{k_i}'$, we use the expression of $e_{\Theta_0}=\sum_{m_{l_j}'<m_{l_j}} (m_{l_j}-m'_{l_j})e_{C_{l_j}}+\hbox{others}$ to pair with. Since the irreducible curves involved in the expression are all different from $C_{k_i}$, we have $e_{C_{k_i}}\cdot e_{\Theta_0}\ge 0$. Similarly, for $C_{l_j}$, we use the expression of $e_{\Theta_0'}=\sum_{m_{k_i}<m_{k_i}'}(m_{k_i}'-m_{k_i})e_{C_{k_i}}+\hbox{others}$. We have  $e_{C_{l_j}}\cdot e_{\Theta_0'}\ge 0$. For $C_{q_p}$, we could use either $e_{\Theta_0}$ or $e_{\Theta_0'}$. Since $e_{\Theta_0}=e_{\Theta_0'}$, we have $(e-e_{\Theta_0})\cdot e_{\Theta_0}\ge 0$. 

Moreover, we have the strict inequality $e\cdot e_{\Theta_0}> e_{\Theta_0}^2$. This is because we assume the original $\Theta, \Theta'$ have at least one common component and because they are connected  by Theorem 1.5 of \cite{LZrc}. The first fact implies there is at least one index in $k_i$, $l_j$ or $q_p$. The second fact implies at least one of the intersection of $C_{k_i}$, $C_{l_j}$ or $C_{q_p}$ with $e_{\Theta_0}$ as in the last paragraph would take positive value.

The inequality  $e\cdot e_{\Theta_0}> e_{\Theta_0}^2$ implies there are more intersections  than the homology intersection number $e_{\Theta_0}^2$ of our new subvariety $\Theta_0$ and $\Theta_0'$. This contradicts to the  positivity of local intersection and the fact that $\Theta_0, \Theta_0'$ have no common component. Hence $\Theta=\Theta'$.

We will use  $C_{x_1, \cdots, x_l}$ to denote the unique subvariety passing through the $l$ points $\{x_1, \cdots, x_l\}$. Apparently, changing the order of $x_i$ gives the same curve. Thus we obtain a well-defined map $h: (x_1, \cdots, x_l)\mapsto C_{x_1, \cdots, x_l}$ from $\hbox{Sym}^l  S^2\cong \mathbb CP^l$ to $\mathcal M_e$.

Since $S$ is $J$-holomorphic and $e'\cdot e=e^2+1=l>0$, 
any curve in $\mathcal M_e$ intersects with $S$ at exactly $l$ points by the positivity of local intersection of distinct irreducible $J$-holomorphic subvarieties. 
Therefore $h$  is one-to-one and  surjective.

Now let us show that $h$ is a homeomorphism, namely both $h$ and $h^{-1}$ are continuous. 
Since $\hbox{Sym}^l  S^2$ is Hausdorff  and $\mathcal M_e$ is compact, 
if we can  show that $h^{-1}:\mathcal M_e\to \hbox{Sym}^l  S^2$ is continuous, it follows that $h$ is also continuous.
 To show $h^{-1}$  is continuous, consider a sequence $C_i \in \mathcal M_e$ approaching to its Gromov-Hausdorff limit $C$. Let the intersection of $C_i$ (resp. $C$) with $S$ be $(x_1^i, \cdots, x_l^i)$ (resp. $(x_1, \cdots, x_l)$). Then $(x_1^i, \cdots, x_l^i)$ has to approach $(x_1, \cdots, x_l)$ by the first item of the definition of topology on $\mathcal M_e$. Therefore $h$ is a homeomorphism.  
\end{proof}

We remark that the subvarieties determined by $(x_1^i, \cdots, x_l^i)$ with $x_j^i\in M$ will not converge in general, especially when two points in the tuple converge to same point by a simple dimension counting. However when $x_i\in S$, they indeed converge in the Gromov-Hausdorff sense since the tangent plane is fixed as that of $S$.

Notice that since the configuration of a subvariety of a sphere class is a tree \cite{LZrc}, a non-primitive class will never be an irreducible component in a reducible subvariety. This is because there will be another irreducible component intersecting the non-primitive class more than once to form cycles, since the reducible subvariety is connected.

Usually, nefness is a numerical way to guarantee the existence of smooth $J$-holomorphic curves.

\begin{lemma}\label{nef=>smooth}
If $e$ is a $J$-nef class with $g_J(e)=0$ and $\mathcal M_e\ne \emptyset$, then it is represented by a smooth $J$-holomorphic sphere of non-negative self-intersection. 
\end{lemma}
\begin{proof}
Since $e$ is $J$-nef and $J$-effective, we have $e\cdot e\ge 0$. Then by Proposition 4.5 of \cite{LZ-generic}, we know $e$ is represented by a smooth rational curve.
\end{proof}

\begin{cor}
Suppose $J$ is a tamed almost complex structure on a rational surface $M$. If $e$ is a $J$-nef class with $g_J(e)=0$ and $\mathcal M_e\ne \emptyset$, then $\mathcal M_e$ is homeomorphic to $\mathbb CP^l$ with $l=e\cdot e+1$.
\end{cor}
\begin{proof}
It is a combination of Lemma \ref{nef=>smooth} and Theorem \ref{homeoCPl}.
\end{proof}

Theorem \ref{homeoCPl} also gives a vanishing result for sheaf cohomology of complex rational surfaces. 
\begin{prop}
Let $M$ be a complex rational surface and 
$D$ a smooth rational curve with $D^2\ge 0$. Then $H^p(M, \mathcal O(D))=0$ for $p>0$. 
\end{prop}
\begin{proof}

First, $K-D$ is not an effective divisor since $(K-D)\cdot D<0$ and $D$ is a smooth divisor with $D^2\ge 0$. Hence, $H^2(M, \mathcal O(D))=H^0(M, \mathcal O(K-D))=0$. Moreover, for $p>2$, $H^p(M, \mathcal O(D))=0$ by dimension reason.

To show $H^1(M, \mathcal O(D))=0$, we first compute the Euler characteristic $\chi(D)$. By Riemann-Roch theorem for non-singular projective surfaces,
\begin{equation}\label{RRsurface}
\begin{array}{lll}
 \chi(D)&=&\chi(0)+\frac{1}{2}([D]\cdot [D]-K\cdot [D])\\
&&\\
 &=&\chi(0)+D^2+1\\
&&\\
 &=&\frac{1}{12}(c_1^2+c_2)+D^2+1\\
&& \\
 &=&D^2+2\\
 \end{array}
 \end{equation}
 when $M$ is a rational surface. 
On the other hand, it follows from Theorem \ref{homeoCPl} that $$\dim H^0(M, \mathcal O(D))=l+1=D^2+2.$$ Since $\chi(D)=\dim H^0(M, \mathcal O(D))-\dim H^1(M, \mathcal O(D))+\dim H^2(M, \mathcal O(D)),$ we have $H^1(M, \mathcal O(D))=0$.
\end{proof}

\section{$J$-holomorphic tori}
This section is on the $J$-holomorphic tori, {\it i.e.} a subvariety in a class $e$ with $g_J(e)=1$. The first part is on a  non-associative addition on elliptic curve induced from almost complex structures of the rational surface. In the second part, we will explore the method in last two sections to study the moduli space $\mathcal M_e$. We will explain our method through an example.
\subsection{Non-associative addition on elliptic curve}
In this subsection, we will show that the primitive case of the statement of Theorem \ref{homeoCPl} follows from a more general framework on the generalization of the addition in the elliptic curve theory. Hence, we start with the classical algebraic curve theory. 

For an algebraic curve $C$ of genus $g$, there is a natural map from the symmetric product to the Jacobian of the curve $u: \hbox{Sym}^nC \rightarrow J(C)$ by $$u(p_1, p_2, \cdots, p_n)\rightarrow (\sum_i \int_{p_0}^{p_i}\omega_1, \cdots, \sum_i \int_{p_0}^{p_i}\omega_g),$$ where $\omega_1, \cdots, \omega_g$ form a basis of $H^0(C, K)$. When $n\ge 2g-2$, this map is very useful in determining the topology of $\hbox{Sym}^n\Sigma_g$. This matches with the philosophy of our discussion in previous two sections. The elements of symmetric product $\hbox{Sym}^nC$ are just the divisors of degree $n$ on $C$. The subset $\hbox{Pic}^n(C)$ of the Picard group is the isomorphism classes of degree $n$ line bundles. It is just the quotient of $\hbox{Div}^n(C)$ modulo linear equivalence.  Abel's theorem says that the map $u$ factors through $\hbox{Pic}^n(C)$ and the induced $\phi: \hbox{Pic}^n(C)\rightarrow J(C)$ is a bijection.

When we take the curve $C$ to be an elliptic curve, then the Jacobian is identified with the elliptic curve with its addition structure. In particular, the map $u$ is now a map $$u: \hbox{Sym}^n(C)\rightarrow C, \, \, (p_1, \cdots, p_n)\mapsto p_1+\cdots +p_n,$$where the addition is the one of the elliptic curve $C$. We first show that, when $n>0$, the map is surjective. The canonical divisor $K$ is trivial and $h^0(K-D)=0$ for any effective divisor $D$. By Riemann-Roch theorem $h^0(D)=n-g+1+h^0(K-D)=n>0$. Hence $u$ is surjective. The preimage $u^{-1}(c)$ for a point $c\in C$ is the set of all possible $(p_1, \cdots, p_n)\in \hbox{Sym}^n(C)$ with $p_1+\cdots +p_n=c$. By above calculation $h^0(D)=n$, hence the map $u$ is a holomorphic fibration over the elliptic curve $C$ with fiber $\mathbb CP^{n-1}$. In particular, we have shown the following, which we want to generalize it to our setting later.

\begin{lemma}\label{ellcurv}
The space of unordered $n$-tuples $(y_1, \cdots, y_{n})$ on an elliptic curve with $y_1+\cdots+y_{n}=0$ is homeomorphic to the space of unordered $n-1$-tuples $(x_1, \cdots, x_{n-1})$ on a rational curve, i.e. to $\hbox{Sym}^{n-1}S^2=\mathbb CP^{n-1}$.
\end{lemma}

Topologically, we have shown the $\hbox{Sym}^n(T^2)$ is a $\mathbb CP^{n-1}$ bundle over $T^2$. Moreover, it can be shown that the fibration is non-trivial. For simplicity of notation, we argue it for $n=2$. A section of $u$ is given by $c\mapsto (x, c-x)$ where $c$ varies in the base $C$. We take two such sections corresponding to $x$ and $x'$ with $x\ne x'$. Then they intersect at only one point corresponding to the value $c=x+x'$. Hence $\hbox{Sym}^2(T^2)$ is the nontrivial $S^2$ bundle over $T^2$.

Let it digress a bit more on the symmetric product to put the results of previous sections into our current discussion. In fact, when $n$ is small there are less  rational curves embedded in $\hbox{Sym}^n\Sigma_g$. For example, when $n=1$, it is clear that there is no embedded sphere unless $g=0$. 
When $n=2$ and $g>1$, $\hbox{Sym}^2\Sigma_g$ always has an embedded symplectic sphere. It could be seen by choosing a hyperelliptic structure $\tau$ on $\Sigma_g$. Then $h^0(p+\tau(p))=2$ which provides a $\mathbb CP^1$. Hence, when $C$ is hyperelliptic, the map $u$ has a fiber $\mathbb CP^1$ and each other fiber a single point. This rational curve has self-intersection $1-g(C)$. Since, a $-1$ curve has non-trivial Gromov-Witten invariant and $T^4$ does not admit any such curve, this is one way to see any genus $2$ curve is hyperelliptic. When $C$ with $g(C)\ge 3$ is not hyperelliptic, this $\mathbb CP^1$ is not holomorphic and is mapped to the image in its Jacobian.
 When $n$ gets larger, we have more non-trivial linear systems which give us embedded projective spaces in the symmetric product. 
 
Now we end the digression and generalize the above discussion to our setting.

 As mentioned in the proof of Theorem \ref{homeoCPl}, all non-primitive sphere classes are Cremona equivalent to $2H$ in $\mathbb CP^2\#k\overline{\mathbb CP^2}$, hence of self-intersection $4$ and $l=5$. Hence, for the convenience of notation, we will write $e=2H$ in the following. When $k<9$, we can apply the classification of negative rational curves in \cite{p=h} to identify all possible subvarieties in class $2H$ and show the moduli space is $\mathbb CP^5$. 

When $k\ge 9$, one might hope to use a method similar to what was done for primitive classes. Since both $2H$ and $H$ are $J$-nef classes with $g_J(e)=0$, we could find smooth rational curves representing $2H$ and $H$. Hence $\{(2H, 1), (H, 1)\}$ is a nodal curve in the case of Corollary 2 of \cite{Sik}. Hence, one could find a smooth (elliptic curve) representative of $3H$. Let it be $S$.

For any $5$ points tuples $(y_1, \cdots, y_5)\in \hbox{Sym}^5S$, there is a rational curve in class $e$ passing through it. However $e\cdot [S]=6$ and each intersection is positive, we have to show that the sixth intersection with $S$ (possibly coinciding with one of $y_i$) is fixed by the first five. In other words, we want to show there is a unique (possibly reducible) rational curve $C_{y_1, \cdots, y_5}$ in class $e$ passing through $(y_1, \cdots, y_5)$. Although $S$ is an elliptic curve, the class $2H$ is spherical. Hence, the corresponding argument in Theorem \ref{homeoCPl} confirms the uniqueness of the rational curve $C_{y_1, \cdots, y_5}$.

When $J$ is an integrable complex structure, it is the topological interpretation of the abelian group structure of elliptic curves: for any divisor $(y_1, \cdots, y_5)$ of an elliptic curve, we can canonically choose the sixth point as $-(y_1+\cdots+y_5)$. Hence, the intersection points correspond to an unordered $6$-tuple $(y_1, \cdots, y_6)$ on an elliptic curve with $y_1+\cdots+y_6=0$. A conical curve determines such a $6$-tuple on $S$ by its intersection with the elliptic curve $S$ as we will show that it follows from Cayley-Bacharach theorem. On the other hand, any such $6$-tuple determines the cubic curve uniquely as we show in the last paragraph. Since the conical curves, {\it i.e.} the curves in class $2H$, form a linear system of dimension $5$, we have shown such $(y_1, \cdots, y_6)$ form a $\mathbb CP^5$. This is a special case and another interpretation of Lemma \ref{ellcurv}.

In particular, we obtain a map $h: (x_1, \cdots, x_5) \mapsto (y_1, \cdots, y_6)\mapsto C_{y_1, \cdots, y_5}$ from $\hbox{Sym}^5 S^2\cong \mathbb CP^5$ to $\mathcal M_e$. It is clear $h$ is surjective, injective and it is continuous. 

When $J$ is a non-integrable almost complex structure, we could define a commutative addition by the exactly same way. Recall that $[S]=3H$. Then any rational curves in class $H$ would intersect $S$ at three points. Let them be $z_1, z_2, z_3$. Since two points are enough to determine the rational curve in class $H$, we could define a symmetric function $z_3=f_1^-(z_1. z_2)$, which could be viewed as the negative of an ``addition". 
If we further take a point $O\in S$ to be the ``zero", we are able to determine the addition $f_1(z_1, z_2)=f_1^-(f_1^-(z_1, z_2), O)$. Notice, in the integrable case the zero point has to be an inflection point. But we do not require this to be true, since we do not require $f_1^-(O, O)=O$. 
Similarly, we could associate another symmetric function $f_2$ to five points on $S$: $z_6=f_2^-(z_1, \cdots, z_5)$ where $z_1, \cdots, z_6$ are intersection points of a conic with $S$, and $f_2(z_1, \cdots, z_5)=f_1^-(f_2^-(z_1, \cdots, z_5), O)$. The functions $f_1, f_2$ could be viewed as deformations of addition structure on the elliptic curve $S$.

However, in general, this new ``addition" $f_1$ is not associative albeit commutative, thus gives rise only a loop structure instead of an abelian group. Even if we have shown that the moduli space $\mathcal M_H=\mathbb CP^2$ for any tamed almost complex structure, we do not have an authentic linear system structure in general.   While our desired result $\mathcal M_{2H}=\mathbb CP^5$ follows from the generalization of Lemma \ref{ellcurv} for $n=6$ by replacing the addition by the function $f_2$. Notice that $f_2$ is not determined by $f_1$ in general, although it is true for the integrable case as we will explain in a moment. More generally, we expect a generalization of the fibration from $\hbox{Sym}^6(T^2)$ to $T^2$ where the projection is given by $f_1^-(f_2(x_1, \cdots, x_5), x_6)$.

Return to integrable complex structure, the associativity of the addition is equivalent to

 \smallskip

{\bf Cayley-Bacharach theorem:} {\it If two (possibly degenerate) curves in $3H$ intersect at $9$ points, then any other curve passing through $8$ of them also passes through the ninth point, where intersection points are counted with multiplicities.} 

 \smallskip

In fact, we first choose points $x, y, z$ and the zero. Then by drawing lines, the intersections with the elliptic curve would give other $4$ points $x+y, -(x+y), y+z, -(y+z)$. The ninth point has two expressions, $(x+y)+z$ or $x+(y+z)$, from two different sets of $3$ lines. Then the associativity follows from the Cayley-Bacharach theorem.

This result has many other implications. In particular, it could make our previous discussion more accurate. Precisely, the addition on the elliptic curve $S$ is determined by its intersection with lines. Namely, if $S$ intersects a line in three points $x_1, x_2, x_3$, then $x_1+x_2=-x_3$. A conical curve $C$ intersect $S$ in six points $x_1, \cdots, x_6,$ and we can show $x_1+\cdots+x_6=0$. Namely,  the line $L_1$ passing through $x_1, x_2$ intersect $S$ at a third point, say $x_7$. Similarly, the line passing through $x_3, x_4$ would intersect $S$ at another point $x_8$. The line $L_3$ passing through $x_7, x_8$ intersect $S$ at another point $x_9$. Finally, there is a line $L_4$ passing through $x_5, x_9$ and intersects $S$ with another point. Cayley-Bacharach theorem implies that the point is $x_6$: the two set of curves $\{L_1, L_2, L_4\}$ and $\{C, L_3\}$ are both in class $3H$ and pass through $8$ points $x_1, \cdots, x_5, x_7, x_8, x_9$, hence $\{L_1, L_2, L_4\}$ would also intersect $S$ at the ninth intersection point, $x_6$, of $\{C, L_3\}$ and $S$. This implies $$x_6=-(x_5+x_9)=-x_5+x_7+x_8=-x_5-(x_1+x_2)-(x_3+x_4)=-(x_1+\cdots +x_5).$$

The Cayley-Bacharach theorem is not true for general tamed almost complex structure as we see above.  By dimension counting, the expected dimension of $\mathcal M_{3H}$ is $9$, hence there will be only finite many curves passing through $9$ points for a generic almost complex structure.

It will be interesting to know whether this actually gives another criterion of integrability of almost complex structures. 

\begin{question}\label{cri_int}
Assume the two-variable symmetric function $f_1$ is associative, {\it i.e.} $f_1(x, f_1(y, z))=f_1(f_1(x, y), z)$ for $x, y, z\in S$, is it true that the almost complex structure is integrable?
\end{question}

In fact, we expect to express $f_1$ as a perturbation of addition with the extra terms determined by the Nijenhuis tensor. Apparently, Question \ref{cri_int} holds only when the zero $O$ is chosen as an inflection point, {\it i.e.} a point in $S$ which has a third order contact with a rational curve in class $H$.

\subsection{Moduli space of tori: a case study}\label{tori}
In this subsection, we would like to show that the method applied in last two sections could be used to study the moduli space $\mathcal M_e$ when $g_J(e)=1$ and $e$ is $J$-nef. Here we only analyze a single example, {\it i.e.} $M=\mathbb CP^2\#8\overline{\mathbb CP^2}$ and $e=-K=3H-E_1-\cdots -E_8$. The method could be pushed to more general cases. Some of the discussion in \cite{Loo} might be useful for our discussion. However, all the subvarieties in \cite{Loo} are reduced, while we are also working with general non-reduced subvarieties, {\it i.e.} we allow multiplicities.

We assume the class $-K$ is $J$-nef. 
We discuss the possible singular subvarieties in class $-K$ of $\mathbb CP^2\#8\overline{\mathbb CP^2}$. First, we will show that the combinatorial behaviour of a reducible variety will not be too bad.

\begin{lemma}\label{g=1red}
Suppose $e$ is a $J$-nef class with $g_J(e)=1$. If $\Theta=\{(C_i, m_i)\}_{i=1}^n$ is a connected subvariety in class $e$ with $g_J(e_1)=1$ and $e_1^2\ge 0$, then $C_i$ are rational curves for $2\le i\le n$, and $$\sum_{i=1}^nm_il_{e_i}\le l_e-1.$$
\end{lemma}
\begin{proof}
By Theorem 1.4 of \cite{LZrc}, $\sum g_J(e)\ge \sum g_J(e_i)$. Hence there is at most one $C_i$ has $g_J(e_i)=1$, others have $g_J=0$.  This component is just our $C_1$.

Recall that $l_e=\max\{\frac{1}{2}(e\cdot e-K_J\cdot e), 0\}$. It is $\max\{e^2, 0\}$ when $g_J(e)=1$, and $\max\{e^2+1, 0\}$ when $g_J(e)=0$.

Use $1, \cdots, k$ to label the irreducible components whose classes have self-intersection at least $0$. In particular, $e_1$ is just the one in our statement satisfying $g_J(e_1)=1$ and $e_1^2\ge 0$. 
Notice   $l_{e_i}=0$ for $i=k+1, \cdots, n$.

Since $\Theta$ is connected, $e_j\cdot (e-m_je_j)\geq 1$ for each $j$.
Therefore $l_e$ can be estimated as follows:
\begin{equation} \label{3terms}
\begin{array}{lll}
 l_e&=&e\cdot e\\
&&\\
 &=&\sum_{j=1}^k(m_j^2e_j\cdot e_j+ m_je_j\cdot (e-m_je_j))  + \sum_{i=k+1}^n m_ie_i\cdot e\\
&&\\
 &\geq&1+\sum_{j=1}^k m_jl_{e_j}+ \sum_{i=k+1}^n m_ie_i\cdot e\\
&& \\
 &=& 1+\sum_{j=1}^n m_jl_{e_j} + \sum_{i=k+1}^n m_ie_i\cdot e.\\
 \end{array}
 \end{equation}
The number $1$ appears in the inequality since $l_{e_1}=e_1^2$. 
Since $e$ is $J$-nef, we have $l_e\ge 1+\sum m_jl_{e_j}$.
\end{proof}

The following proposition describes reducible varieties in the $J$-nef class $-K$.
\begin{prop}\label{-Kred}
Any subvariety in $\mathcal M_{-K}$ is connected. All irreducible components of a subvariety $\Theta=\{(C_i, m_i)\}\in\mathcal M_{red, -K}$ are smooth rational curves. Moreover, they are of negative self-intersection.
\end{prop}
\begin{proof}
Suppose there is a disconnected variety $\Theta=\cup \Theta_i$, where $\Theta_i$ are connected components. Hence $-K=\sum e_{\Theta_i}$. Since $-K$ is $J$-nef,  we have $-K\cdot e_{\Theta_i}\ge 0$. Since $(-K)^2=1$, we know $-K\cdot e_{\Theta_i}=e_{\Theta_i}^2=0$ or $1$. Since $-K=3H-E_1-\cdots -E_8=\sum e_{\Theta_i}$, we know there is a $\Theta_i$ such that $e_{\Theta_i}\cdot H>0$. Then the argument of Lemma 4.7(2) of \cite{p=h} implies $e_{\Theta_i}^2=1$ and actually $e_{\Theta_i}=-K$. Hence $\Theta$ has only one connected component. Thus any element in $\mathcal M_{-K}$ is connected.

Hence, by Theorem 1.4 of \cite{LZrc}, we have at most one $C_i$, say $C_1$, has $g_J(e_{C_1})=1$. Moreover, by Lemma \ref{g=1red}, if $g_J(e_{C_1})=1$ and $e_{C_1}^2\ge 0$, then we have $0=l_{-K}-1\ge \sum_{i=1}^nm_il_{e_{C_i}}$. Hence, all $l_{e_i}=0$. In particular, $e_{C_1}^2\le l_{e_{C_1}}=0$. Hence we have $e_{C_1}^2\le 0$ in any case. However, this contradicts to Lemma 4.7(2) of \cite{p=h}. Therefore, all irreducible components are rational curves.

Now, we can similarly argue as \eqref{3terms} for $e=-K$
\begin{equation}\label{1>0+0}
\begin{array}{lll}
 1&=&e\cdot e\\
&&\\
 &=&\sum_{j=1}^k(m_j^2e_j\cdot e_j+ m_je_j\cdot (e-m_je_j))  + (\sum_{i=k+1}^n m_ie_i\cdot e)\\
&&\\
 &\geq&\sum_{j=1}^k m_jl_{e_j}+ (\sum_{i=k+1}^n m_ie_i\cdot e)\\
&& \\
 &=& \sum_{j=1}^n m_jl_{e_j} + (\sum_{i=k+1}^n m_ie_i\cdot e).\\
 \end{array}
 \end{equation}
Hence at most one index, say $1$, has $e_1^2\ge 0$. That is, $k=1$. This happens only when the equality of \eqref{1>0+0} holds. Moreover, $e_1^2=0$ and $m_1=1$. Hence we know $K\cdot e_1=-2$ by adjunction formula. In particular, this implies $e_1\cdot (e-e_1)=2$. This inequality would prevent the equality in \eqref{1>0+0} holding. Thus all the irreducible components have negative self-intersection.
\end{proof}

In fact, since $-K$ is assumed to be $J$-nef, there are only a few possibilities for the negative curve classes. If a curve $C$ has $e_C^2<0$, then $K\cdot e_C+e_C^2\ge -2$ and $K\cdot e_C\le 0$ imply $C$ is a rational curve with $e_C^2=-1$ or $e_C^2=-2$. 

The main obstacle to generalize the rational curve argument is that $l_e$ points no longer determine a unique curve in class with $g_J(e)\ge 1$. However, the following lemma  gives us a bit room to extend our strategy in last two sections.

\begin{lemma}\label{uniquesingular}
Let $\Theta\in \mathcal M_{-K}$. If $x\in |\Theta|$ is a singular point. Then $\Theta$ is the unique element of $\mathcal M_{-K}$ passing through $x$.
\end{lemma}
\begin{proof}
Assume there is another $\Theta'\in \mathcal M_{-K}$ intersecting $\Theta$ at $x$. If they do not share irreducible components passing through $x$, then $x$ would contribute at least $2$ to the intersection number of $\Theta$ and $\Theta'$. However, it contradicts to the local positivity of irreducible $J$-holomorphic curves, since $e_{\Theta}\cdot e_{\Theta'}=(-K)^2=1$.

If $\Theta$ and $\Theta'$ share irreducible components passing through $x$, then they are both in $\mathcal M_{-K}^{red}$. By Proposition \ref{-Kred}, each component is a rational curve. Since $x$ is a singular point, it must be the intersection point of at least two irreducible components of $\Theta$. Then we can apply the same argument of  Lemma \ref{uniquereducible} to get two cohomologous subvarieties $\Theta_0$ and $\Theta_0'$ with no common components. If $x\in \Theta_0$, we have  $1\ge (-K)\cdot e_{\Theta_0}>e_{\Theta_0}^2$. If $(-K)\cdot e_{\Theta_0}=0$, we have $e_{\Theta_0}\cdot e_{\Theta_0'}=e_{\Theta_0}^2<0$, contradicting to the local positivity of intersection. Hence $-K\cdot e_{\Theta_0}=1$, or equivalently $-K\cdot (-K-e_{\Theta_0})=0$. In other words, all the removed components are $-2$ rational curves. However, in this case $(-K-e_{\Theta_0})\cdot e_{\Theta_0}=\sum m_ie_{C_i}(-K-\sum m_ie_{C_i})>0$. Since all $C_i$ are $-2$ rational curves, we have $$\sum m_ie_{C_i}\cdot (-K-\sum m_ie_{C_i})=-(\sum m_ie_{C_i})^2=2\sum m_i^2-2\sum m_im_je_{C_i}\cdot e_{C_j}$$ is an even number. Hence $1= (-K)\cdot e_{\Theta_0}\ge e_{\Theta_0}^2+2$, which implies the impossible relation $e_{\Theta_0}^2< 0$ again.

Hence $\Theta$ is the unique subvariety in class $-K$ passing through the point $x$.
\end{proof}

From now on, we will further assume that there is no irreducible $J$-holomorphic curves with cusp singularity in $\mathcal M_{-K}$. Notice there is a confusion of irreducible cuspidal curves: From our subvariety viewpoint, they are tori, since $g_J=1$; While from the $J$-holomorphic map viewpoint, they are rational curve, since the model curve is a sphere. By \cite{B}, all such almost complex structures form an open dense set. Under our assumption, an irreducible curve in $\mathcal M_{-K}$ has at most nodal singularities.

The main point to prevent the cusps is the unobstructedness result.

\begin{prop}\label{unobstructedtori}
Let $C$ be an irreducible nodal curve in class $-K$. There is a local homeomorphism $(\mathcal M_{-K}, C)\rightarrow (\mathbb C, 0)$.  
\end{prop}
\begin{proof}
Essentially, it is due to \cite{Sik}, which proves that the moduli space of $J$-holomorphic maps is smooth at an immersed $J$-curve with $K\cdot e_C<0$, and locally the moduli space is $(\mathbb C, 0)$.
In our situation, an irreducible nodal curve is the image of an immersion $\phi$.  As we recalled in the beginning of Section 2.1, if $C$ is irreducible, it determines the map $\phi$ up to automorphisms. Moreover, $e_C=-K$ thus $K\cdot e_C=-1<0$. Hence, the result applies to our case. In particular, the neighborhood of $\phi$, which is identified with our $\mathcal M_{-K}$ locally at $C$, is homeomorphic to $\mathbb C$ with all elements in $\mathbb C\setminus \{0\}$ smooth $J$-holomorphic maps. Hence the conclusion holds.
\end{proof}

Besides irreducible nodal curves, the remaining elements in $\mathcal M_{-K}$ are reducible varieties whose irreducible components are all smooth rational curves. These rational curves are of self-intersection $-1$ or $-2$. There are only finitely many such curves, see {\it e.g.} Proposition 4.4 of \cite{p=h}. Hence $\mathcal M_{red, -K}$ is a set of isolated points in $\mathcal M_{-K}$. However, we do not know if it is locally Euclidean. Nonetheless, we have  the path connectedness of $\mathcal M_{-K}$.

\begin{prop}\label{MKconn}
Suppose there is a singular subvariety in class $-K$. Then the moduli space $\mathcal M_{-K}$ is path connected when there is no irreducible $J$-holomorphic cuspidal curve in class $-K$.
\end{prop}
\begin{proof}
The Gromov compactness implies the moduli space $\mathcal M_{-K}$ is compact. In particular, there are finitely many irreducible nodal curves by Proposition \ref{unobstructedtori}, and finitely many reducible subvarieties by the discussion in the last paragraph. Since $l_{-K}=1$, by Seiberg-Witten theory, there is an $J$-holomorphic subvariety passing through any given point. However, such a subvariety might not be unique. 

We will show that a smooth subvariety in $\mathcal M_{-K}$ is connected to any singular subvarieties in $\mathcal M_{-K}$ by paths. Since there are finitely many singular subvarieties, the complement $M'$ of their support in $M$ is path-connected. For a smooth subvariety $C$, choose a point $x\in C$ which is also in $M'$. Then for any singular subvariety $\Theta$, choose a singular point $y\in |\Theta|$. Take a path $P(t)$ such that  $P(0)=x$, $P(1)=y$ and $P[0, 1)\subset M'$. We look at the set $$T=\{t\in [0, 1]| C(t)\hbox{ passes through }P(t),\hbox{ and is connected to }C \hbox{ by path}\},$$ where $C(t)\in \mathcal M_{-K}$ and $C(0)=C$.
Because of our choice of $P(t)$, each $C(t)$, $t\in [0, 1)$, is a smooth curve. By Lemma \ref{uniquesingular}, $C(1)=\Theta$ since $y$ is a singular point of $\Theta$. 

 The set $T$ is non-empty because $0\in T$, open because Proposition \ref{unobstructedtori}, and closed because of Gromov compactness. Hence $T=[0, 1]$ and we have constructed a path $C(t)\in \mathcal M_{-K}$ from a smooth subvariety $C$ to a singular subvariety $\Theta$. Hence $\mathcal M_{-K}$ is path connected. 
\end{proof}

The following is Theorem \ref{introtori}.

\begin{theorem}\label{smoothconn}
If there is an irreducible (singular) nodal curve in $\mathcal M_{-K}$, then $\mathcal M_{smooth, -K}$ and $\mathcal M_{-K}$ are both path connected.
\end{theorem}
\begin{proof}
The path connectedness of the moduli space $\mathcal M_{-K}$ follows directly from Proposition \ref{MKconn}.

The space $\mathcal M_{smooth, -K}$ is the subset of $\mathcal M_{-K}$ where all the elements are smooth curves. By Proposition \ref{MKconn}, for any smooth curves $C, C'\in \mathcal M_{smooth, -K}$, we have paths $C(t), C'(t)\subset \mathcal M_{-K}$ with $C(0)=C, C(0)=C', C(1)=C'(1)=C_{nd}$ and $C([0, 1)), C'([0, 1))\subset \mathcal M_{smooth, -K}$. Here, $C_{nd}$ is an irreducible singular nodal curve. By Proposition \ref{unobstructedtori}, there is a locally Euclidean  neighborhood $U$ of $C(1)=C'(1)$ in $\mathcal M_{-K}$. Hence for $t$ close to $1$, $C(t), C'(t)\in U$. Since $(U, \{C(1)\})$ is homeomorphic to $(\mathbb R^2, 0)$, $U\setminus \{C(1)\}$ is homeomorphic to $\mathbb R^2\setminus \{0\}$ and $C(t)$ and $C'(t)$ are connected by a path in it.  Hence $C$ and $C'$ are connected by a path in $\mathcal M_{smooth, -K}$.
\end{proof}

Our method does not apply to the case when $\mathcal M_{smooth, -K}=\mathcal M_{-K}$, {\it i.e.} when all subvarieties in class $-K$ are smooth. However, we believe it cannot happen.

\begin{question}\label{1sing}
For any tamed $J$, do we always have a singular $J$-holomorphic subvariety in class $-K$?
\end{question}

In the integrable case, this is apparently correct. Moreover, the total number of singular points for the curves in class $-K$ is $12$ by a general Euler characteristic argument. Namely, all the curves in class $-K$ form a pencil. After blowing up at the common intersection, we have the universal curve $\mathcal C\rightarrow \mathcal M_{-K}=S^2$ which is diffeomorphic to $M\#\overline{\mathbb CP^2}=\mathbb CP^2\#9\overline{\mathbb CP^2}$. The Euler number $$e(\mathcal C)=e(S^2)\cdot e(T^2)+\# \hbox{ singular points}=\# \hbox{ singular points}$$ and $e(\mathbb CP^2\#9\overline{\mathbb CP^2})=12$. Hence we have $12$ singular points in total for the curves in class $-K$.

\section{More applications}
In this section, we give several applications on spaces of tamed almost complex structures and symplectic isotopy. In particular, we exhibit the exotic behaviour of subvarieties of rational surfaces in a sphere class. In particular, Example \ref{connectedgenus3} gives a connected subvariety with a genus $3$ component in an exceptional curve class. Moreover, the graph corresponding to the subvariety has a loop.
\subsection{Spaces of tamed almost complex structures}
It is known that the space $\mathcal J^{\omega}$ of $\omega$-tamed almost complex structures is contractible, thus path connected. 
Next we define $\mathcal J_{e-nef}\subset \mathcal J^{\omega}$.
\begin{definition}\label{enef}
Suppose $e\in H^2(M, \mathbb Z)$. An $\omega$-tamed $J$ is in $\mathcal J_{e-nef}$ if $e$ is $J$-nef.
\end{definition}

This subspace is also path connected.
\begin{lemma}\label{enefconnected}
If $e$ is represented by a smooth $\omega$-symplectic sphere with non-negative self-intersection, 
then $\mathcal J_{e-nef}$ is path connected.
\end{lemma}
\begin{proof}
The assumption tells us $\mathcal J_{e-nef}\neq\emptyset$ and the ambient manifold is rational or ruled. Then the conclusion basically follows from a well-known argument \cite{MS, Shev}. It is known that $\mathcal J_{reg}$ is path connected by the argument of Theorem 3.1.5 of \cite{MS}. Recall that the subset $\mathcal J_{reg}\subset \mathcal J^{\omega}$ is defined as the set of almost complex structures such that all $J$-holomorphic maps $\phi: \Sigma\rightarrow M$ are regular (or the Fredholm operator $D_{\phi}$ is onto). When $S$ is a smooth $J$-holomorphic sphere, $J$ is regular with respect to the class $e=[S]$ if and only if $e\cdot e\ge -1$. And we know that for $J\in \mathcal J_{reg}$, $\mathcal M_{irr, e, J}\neq \emptyset$ (see e.g. \cite{MS, T1}). Hence $e$ is $J$-nef.
This implies $\mathcal J_{reg}\subset \mathcal J_{e-nef}$. 
This observation ensures us to apply the same argument of Theorem 3.1.5 in \cite{MS}. Namely, the projection of $\cup \mathcal M_{e-non-nef, J}\rightarrow \mathcal J^{\omega}$ is Fredholm with index at most $-2$. 
\end{proof}

When $J$ is tamed, there are some occasions that $\mathcal J_{e-nef}=\mathcal J^{\omega}$. For example, as we have shown in Proposition \ref{smoothT}, $e=T$ in an irrational ruled surface is such a class. Another such example is $e=H-E$ in $\mathbb CP^2\#\overline{\mathbb CP^2}$. In general, $\mathcal J_{e-nef}\ne \mathcal J^{\omega}$ however. For example, when $e=11H-4E_1-\cdots-4E_7-3E_8$, any $J$ on $M=\mathbb CP^2\#8\overline{\mathbb CP^2}$ such that there are smooth curves in classes $3H-2E_1-E_2-\cdots -E_8$ and $3H-E_1-\cdots -E_8$ will do. For such a $J$, the sphere classes (11)-(15) in the list of Proposition 4.7 of \cite{p=h}, {\it e.g.} $e=11H-4E_1-\cdots-4E_7-3E_8$, are not $J$-nef.
In fact, this example works more generally.

\begin{prop}\label{T2comp}
For any $K$-spherical class $e$ 
of $\mathbb CP^2\#k\overline{\mathbb CP^2}, k\ge 8$ and $e\cdot e\ge -1$, there is an integrable complex structure $J$ such that there are subvarieties in $\mathcal M_e$ with an elliptic curve irreducible component.
\end{prop}

\begin{proof}
We start with $k=8$.

First, we discuss a special $-1$ curve class. In the first construction of section 4.2 in \cite{AZ}, we constructed an elliptic $E(1)$ with a $I_5$ fiber and seven $I_1$ fibers. The homology classes of the components of $I_5$ are $H-E_1-E_6-E_7, H-E_2-E_8-E_9, H-E_1-E_2-E_3, E_2-E_5$ and $E_1-E_4$. And there are $7$ disjoint $-1$ sections in classes $E_3, \cdots, E_9$. Since $(H-E_2-E_8-E_9)\cdot E_9=1$, after blowing down $E_9$, we have a complex structure $J_0$ on $\mathbb CP^2\#8\overline{\mathbb CP^2}$ such that there are subvarieties $\{(E_1-E_4, 1), (F, 1)\}\in \mathcal M_E$ with $e_F=-K$ and the $-1$ class $E=3H-E_2-E_3-2E_4-E_5-\cdots-E_8$, where $F$ are induced from smooth fibers of the above elliptic fibration. Notice our $-1$ class $E$ is Cremona equivalent to $E_1$, hence sphere representable.

Then we claim that for any non-negative $K$-spherical class $S$ with $S\cdot E=0$, $\mathcal M_S$ also contains an element with elliptic curve irreducible component.  By adjunction $g_{J_0}(S-E)=0$. We also have $(S-E)^2\ge -1$, $\dim_{SW}(S-E)\ge 0$ and $\dim_{SW}(S)>0$. Hence $SW(S-E)\ne 0$, see {\it e.g.} Proposition 2.3 of \cite{p=h}, and there is a $J_0$-holomorphic subvariety $\Theta_1$ in class $S-E$.  However, $$\Theta=\{(E_1-E_4, 1), (3H-E_1-\cdots -E_8, 1), \Theta_1\}\in \mathcal M_{S},$$ and $g_J(3H-E_1-\cdots -E_8)=1$. 

Now we discuss a general sphere class. 
It is a classical result that the $-2$-rational curve classes (with respect to a fixed canonical class $K$, say $-3H+E_1+\cdots+E_8$) correspond to the root system of the exceptional Lie algebra $E_8$ (this should not be confused with the exceptional class $E_8$). In particular, there are $240$ such classes. The Weyl group $W(E_8)$ is the group of permutations of the roots generated by the reflections in the roots. By Corollary 26.7 of \cite{Man}, $W(E_8)$ acts transitively on the collection of the $-1$-rational curve classes (with respect to a fixed canonical class $K$). Since the canonical class is fixed, the Weyl group action could be realized by a Cremona transformation \cite{LL}, {\it i.e.} a diffeomorphism preserving the canonical class. Hence, for a $-1$ rational curve class $E'$ which is related to $E$ by a diffeomorphism $f$, we know $\mathcal M_{E'}$ would also contain an element with an elliptic curve irreducible component with respect to the complex structure $f^{-1}(J_0)$.

Any $K$-spherical class $e$ of non-negative square is Cremona equivalent to one of the following classes \cite{LBL}

\begin{itemize} 
\item $H-E_1$,
\item $2H$,
\item $H$,
\item $(n+1)H-nE_1$, $n\ge 1$,
\item $(n+1)H-nE_1-E_2$, $n\ge 1$. 
\end{itemize}

In particular, the class $E_8$ is perpendicular to all the above classes. Hence, for any $K$-spherical class $e$, there is an exceptional curve class $E'$ such that $e\cdot E'=0$. 
By the above discussion, there is a Cremona transformation $f$ which transforms $E'$ to an exceptional curve class $E=3H-E_2-E_3-2E_4-E_5-\cdots-E_8$. Hence $f^*(e)\cdot E=0$. Hence, for the tamed almost complex structure $J=f^{-1}(J_0)$, we know $\mathcal M_e$ 
also has the desired property by the above discussion for the sphere classes perpendicular to $E$.

When $k\ge 9$, for any $-1$ curve class $E$, we know there is another $-1$ class $E'$ orthogonal to it. Blow down $E'$, we can use the induction to get the desired result for class $E$. Similarly, for a non-negative sphere class $e$ with $e\cdot E=0$, it is at least orthogonal to another $-1$ curve class $E'$ with $E\cdot E'=0$. Blow down this $E'$, the induction would give the result for $e$. This completes the argument. 
\end{proof}

This proposition implies $\mathcal J_{e-nef}\ne \mathcal J^{\omega}$ for sphere classes on $\mathbb CP^2\#k\overline{\mathbb CP^2}$ with $k\ge 8$.

\begin{cor}\label{spherenon-nef}
When $M=\mathbb CP^2\#8\overline{\mathbb CP^2}$, for any $K$-spherical class $e$ there is an integrable complex structure such that $e$ is not $J$-nef.
\end{cor}
\begin{proof}
When $e\cdot e<0$, the statement follows from gathering information of \cite{p=h}. If $e\cdot H<0$, then $e$ is not $J$-nef for any tamed $J$ since $\mathcal M_H\ne \emptyset$ because $SW(H)\ne 0$. If $e\cdot H=0$, then $e=E_i-\sum_{k_j\ne i}E_{k_j}$ by Lemma 3.5 of \cite{p=h}. Thus, $e\cdot E_i<0$. Since $SW(E_i)\ne 0$, we know $e$ is not $J$-nef. Finally, if $e\cdot H>0$, we have the list in Proposition 4.6 of \cite{p=h}. From the list, we know there is always a $-1$ curve class $E_0$ such that $e\cdot E_0=-1<0$. Since $SW(E_0)\ne 0$, we know $e$ is not $J$-nef.

When $e\cdot e\ge 0$, the statement follows from Proposition \ref{T2comp}. We choose the same $J$ as in the proof of it. 
If on the contrary, $e$ is $J$-nef, then by Theorem 1.5 in \cite{LZrc}, any irreducible component of a subvariety in $\mathcal M_e$ is a rational curve. This contradicts to Proposition \ref{T2comp}.
\end{proof}

We remark that similar construction could lead to other types of interesting examples. For instance, there are examples in classes of $J$-genus $g$ with nontrivial Seiberg-Witten invariant and with genus $g'>g$ irreducible components in a subvariety. 

In previous examples, the subvarieties are disconnected with an elliptic curve component. Below we construct a complex surface such that there is a connected subvariety with a genus $3$ component in an exceptional curve class.

\begin{example}\label{connectedgenus3}
We start with $4$ lines of general position in $\mathbb CP^2$. We choose one of them, say $L_1$, and an intersection point $p=L_2\cap L_3$. We find a genus $3$ $J$-holomorphic curve $C$ in class $4H$ passing through $p$ such that other intersection points with $L_i$ are transversal and are not the intersection points $L_i\cap L_j$. There are $14$ such intersection points.

We blow up once at each of these $14$ points and twice consecutively at $p$.  The latter means that we first blow up at $p$ to get an exceptional curve in class $E_{10}$, then blow up again at the intersection of this curve and the proper transform of $C$ to get an exceptional curve, say in class $E_1$.

This gives us a genus $3$ fibration structure on $\mathbb CP^2\#16\overline{\mathbb CP^2}$. A general fiber has class $4H-E_1-\cdots -E_{16}$. The lines $L_i$ become rational curves in classes $H-E_2-E_3-E_4-E_{16}$,  $H-E_9-E_{10}-E_{11}-E_{12}$, $H-E_{10}-E_{13}-E_{14}-E_{8}$, and $H-E_5-E_6-E_7-E_{15}$.  The curve in class  $E_{10}$ is transformed to a curve in class  $E_{10}-E_{1}$. Below is the configuration of these curves on $\mathbb CP^2\#16\overline{\mathbb CP^2}$.

\begin{center}
\begin{tikzpicture}
\draw (-2, 3)--(1.5,3);
\draw (-1.5, 3.5)--(-1.5, 0);
\draw (0, 3.5)--(0, 0);
\draw (-2.5, 1)--(1.5, 3.5);
 \draw (-2, 1)--(1.5, 1);
 \draw (-0.6, 1.5)--(-0.6, -1);
\draw  (-2.5, -0.5)--(1.6,-0.5);

 \node at (1.8, 0.7){$E_{10}-E_1$};
  \node at (-2.4, 3){$L_1$};
    \node at (-2.7, 1.4){$L_4$};
     \node at (-1.8, 0.5){$L_2$};
   \node at (0.3, 0.5){$L_3$};
    \node at (-0.9, 0.2){$E_1$};
        \node at (1.6, -0.8){$C$};
\end{tikzpicture}
\end{center}

We blow down $5$ sections of the fibration in classes $E_{12}, \cdots, E_{16}$. We choose a curve $C_0$ inherited from the general fiber. Hence $[C_0]=4H-E_1-\cdots-E_{11}$. There exist a curve $C_1$ in class $H-E_2-E_3-E_4$ and a curve $C_2$ in class $H-E_5-E_6-E_7$.

It is straightforward to check that $\Theta=\{(C_0, 1), (C_1, 3), (C_2, 1)\}\in \mathcal M_E$ is a connected subvariety,  where $E=8H-E_1-4E_2-4E_3-4E_4-2E_5-2E_6-2E_7-E_8-\cdots -E_{11}$ is an exceptional curve class in $\mathbb CP^2\#11\overline{\mathbb CP^2}$. Since any $K_J$-spherical class of square $1$ is Cremona equivalent to $H$, we know the class $8H-2E_1-4E_2-4E_3-4E_4-2E_5-2E_6$ is Cremona equivalent to $2H$. In fact, the diffeomorphism could be realized by the composition of two Dehn twists along the $-2$-spheres in classes $2H-E_1-\cdots-E_6$ and $H-E_2-E_3-E_4$. Under this diffeomorphism, class $E$ is transformed to $3H-E_5-E_6-2E_7-E_8-\cdots -E_{11}$, which is representable by a smooth sphere. Hence, $E$ is an exceptional curve class.  

\begin{center}
\begin{tikzpicture}[scale=0.8]
\draw (-2.5,0)--(2.5,0);
\draw (-2.5,0.01)--(2.5,0.01);
\draw (-2.5,0.02)--(2.5,0.02);
\draw (-2, {-sin(60)})--(0.5, {3.5*sin(60)});
\draw (2, {-sin(60)})--(-0.5, {3.5*sin(60)});

  \node at (0, -0.4){$(C_1,3)$};
 \node at (1.5, 1.2){$(C_2,1)$};
 \node at (-1.5, 1.2){$(C_0,1)$};
\end{tikzpicture}
\end{center}

The graph corresponding to the subvariety $\Theta$ has a loop. The curve $C_0$ is a genus $3$ curve with $[C_0]^2=5$. Moreover, $\dim_{\mathbb C} \mathcal M_{[C_0]}=3$ locally since $K\cdot [C_0]=-1<0$.
\end{example}

For a $J$-nef sphere class, it was shown in \cite{LZrc} that any subvariety in it has two features: 1) every irreducible component is a rational curve; 2) the graph of the subvariety is a tree. In the above example, both properties fail.

\begin{example}
We use the same notation as in the previous example. Choose $\Theta'=\{(C_0, 1), (C_1, 2)\}\in \mathcal M_T$, where $T=6H-E_1-3E_2-3E_3-3E_4-E_5-\cdots -E_{11}$ is a class with $g_J(T)=1$, and $SW(T)\ne 0$. Moreover, $\Theta'$ is a connected subvariety with $g_J(C_0)=3>g_J(T)$.
\end{example}

\begin{remark}\label{notGW}
If we look at the space $\overline{\mathcal M}_g(M, J, e)$ of stable $J$-holomorphic curves in the homology class $e$. There is a natural forgetful map from $\overline{\mathcal M}_g(M, J, e)$ to our $\mathcal M_e$ by just looking at the image. In general, this map is not surjective since any irreducible component of an element in the image is the image of a rational curve for a sphere class $e$ by Gromov compactness and meanwhile our above examples show the contrary. In particular, our examples above will not contribute to any Gromov-Witten invariant. 

Take an exceptional class $e=E$ for example. By Gromov compactness, for any tamed $J$, there is always a subvariety in class $E$ whose irreducible components are rational curves. However, the $J$-holomorphic subvarieties in class $E$ might not be unique as in our examples. In fact, for such a $J$, a general member of $J$-holomorphic subvarieties in class $E$ will have a higher genus component.

The map from $\overline{\mathcal M}_g(M, J, e)$ to  $\mathcal M_e$ is also not always injective. For example, when $M=\mathbb CP^2$ and $e=2H$, $\mathcal M_e=\mathbb CP^5$ and $\overline{\mathcal M}_g(M, J, e)$ is the blow up of the space of double lines in $\mathbb CP^5$.
\end{remark}

\subsection{Symplectic isotopy of spheres}
In this section, we will prove Theorem \ref{sphereiso}, which claims  a smooth symplectic sphere $S$ with self-intersection $[S]\cdot [S]\ge 0$ is symplectically isotopic to a holomorphic curve. Here, by symplectic isotopy, we have a two-fold meaning. Let $\mathcal I_{\omega}$ be the space of integrable complex structures  tamed by $\omega$. When $\mathcal I_{\omega}\cap \mathcal J_{e-nef}\ne \emptyset$, $S$ and the holomorphic curve $C$ are symplectic isotopic if they are connected by a path inside the space of smooth $\omega$-symplectic submanifolds. When $\mathcal I_{\omega}\cap \mathcal J_{e-nef}= \emptyset$, $(S, \omega)$ and $(C, I)$ are called symplectic isotopic if there is a path of symplectic form $\omega_t$ with $\omega_0=\omega$ and  $\omega_1=\Omega$ such that $S$ is symplectic with respect to all $\omega_t$, and $S$ and $C$ are symplectic isotopic with respect to $\Omega$.

By \cite{McD}, the ambient manifold $M$ is rational or ruled if $[S]\cdot [S]\ge 0$. Especially, when $[S]\cdot [S]>0$, then $M$ has to be rational, {\it i.e.} $M=\mathbb CP^2\#k\overline{\mathbb CP^2}$ or $S^2\times S^2$. 

\begin{lemma}\label{deformomega}
Let $(M, \omega)$ be a symplectic manifold with $b^+(M)=1$, and $S$ is a smooth symplectic sphere with self-intersection $[S]\cdot [S]\ge 0$. Then we can find a K\"ahler form $\Omega$ and a path of symplectic form $\omega_t$ with $\omega_0=\omega$ and  $\omega_1=\Omega$ such that $S$ is a symplectic submanifold with respect to any $\omega_t$.

Moreover, $[S]$ is represented by a smooth rational curve with respect to some integrable $I$ compatible with $\Omega$.
\end{lemma}
\begin{proof}
First notice that when $b^+=1$ cohomologous symplectic forms are symplectomorphic (see \cite{LL} for example). Hence any symplectic form cohomologous to a K\"ahler form is actually a K\"ahler form.

The result of Theorem 2.7 in \cite{DL} could be restated to adapt to our situation. It says that a cohomology class $e$ is represented by a symplectic form with canonical class $K_{\omega}$ such that $S$ is a symplectic submanifold if  $e^2>0$, $e\cdot E>0$ with all $K_{\omega}$-exceptional spheres $E$ and $e\cdot [S]>0$. Since $S$ is a symplectic sphere of non-negative self-intersection, $[S]$ is represented  by a subvariety for any tamed $J$. Hence any K\"ahler form will be in above mentioned set.  Moreover, since the construction of \cite{DL} is through symplectic inflation, which is a deformation of symplectic structures, then any two such cohomology classes could be connected by a path of symplectic forms. Combining the above two facts, we prove the statement.

For the second statement, we know that there is an integrable complex structure $I$ such that $[S]$ is represented by a smooth rational curve. Then we just choose $\Omega$ such that $(\Omega, I)$ is a K\"ahler structure.
\end{proof}

Notice by \cite{CP}, there are symplectic forms on ruled surface which are not cohomologous to any K\"ahler form (with respect to any integrable complex structure).

Now, let us prove symplectic isotopy of spheres. The result is essentially known, see {\it e.g.} \cite{LW} Proposition 3.2. Here we provide a proof based on our study of $J$-holomorphic subvarieties. 

\begin{theorem}\label{sphereiso}
Any symplectic sphere $S$ with self-intersection $S\cdot S \ge 0$ in a $4$-manifold $(M, \omega)$ is symplectically isotopic to an (algebraic) rational curve. Any two homologous spheres with self-intersection $-1$ are symplectically isotopic to each other.
\end{theorem}

\begin{proof}
Applying Lemma \ref{deformomega}, we first deform the symplectic form $\omega$ to a K\"ahler form $\Omega$, and with $S$ invariant. 
Then we choose an $\Omega$-tamed almost complex structure $J$ on $M$ such that $S$ is a $J$-holomorphic curve. Since $\mathcal I_{\Omega}\cap \mathcal J_{e-nef}\ne \emptyset$ by the second statement of Lemma \ref{deformomega}, we know the existence of a path $\{J_t\}_{t\in [0, 1]}\subset \mathcal J_{e-nef}$ with $J_0=J$, and $J_1=I$ by Lemma \ref{enefconnected}. 

Consider the set $T$ of $\tau\in [0,1]$ so that for every $t\le \tau$ a smooth $J_t$-holomorphic curve $S_t\subset M$ exists that is isotopic to $S$. By definition, $T$ is an interval, and within this interval $S_t$ is symplectically isotopic to $S$. It is an open subset of $[0, 1]$ by the unobstructedness result Theorem \ref{unobstructed}). It remains to show it is closed. Let $t_n\in T$ and $t_n\rightarrow \tau$. 
By Gromov compactness, $S_{t_n}$ converge to a $J_{\tau}$-holomorphic subvariety $\Theta=\{(C_1, m_1),\cdots, (C_k, m_k)\}$.

By Proposition \ref{connected}, we could choose the following path consecutively:

\begin{itemize}
\item The first path $S_{\tau, t}\subset \mathcal M_{e, J_{\tau}}$ with $t\in [0, 1]$ such that $S_{\tau, 0}=\Theta$ and all other $S_{\tau, t}\in \mathcal M_{irr, e, J_{\tau}}$. This is possible because of Proposition \ref{connected}.

\item The second path $S_{t, 1}\subset \mathcal M_{irr, e, J_t}$ with $t\in [\tau-\epsilon, \tau]$. Recall that for a given $J$, the image $\pi_{red, l}(\mathcal M_{red, e, l})\in M^{[l]}$ is a finite union of submanifolds of codimension at least two. Hence, the complement of the union of all these images for $J_t, t\in [\tau-\epsilon, \tau]$, is an open dense set $U\subset M^{[l]}$. Choose an $l$-tuple in $U$ such that all these points are on $S_{\tau, 1}$. If not, we perturb the path in step $1$ to achieve it. By our choice of the $l$-tuple, they determine $J_t$-holomorphic smooth curves $S_{t, 1}$ uniquely for $t\in [\tau-\epsilon, \tau]$.

\item The third path $S_{\tau-\epsilon, t} \subset \mathcal  M_{irr, e, J_{\tau-\epsilon}}$ connects $S_{\tau-\epsilon, 0}=S_{\tau-\epsilon}$ and $S_{\tau-\epsilon, 1}$ in the second path above. This is guaranteed by Proposition \ref{connected}.

\item The last path is the original $S_t$ which connects $S_{t-\epsilon}$ to $S$ through symplectic isotopy.
\end{itemize}
The four paths together ensure $T$ is closed. Hence $T=[0, 1]$. Finally, any K\"ahler structure $I$ is projective and holomorphic curves are algebraic since $p_g=0$ for rational and ruled surfaces. This completes the proof of Proposition \ref{sphereiso} when $[S]\cdot [S]\ge 0$.

When $[S]\cdot [S]=-1$, we have $SW([S])\ne 0$. Hence there is always a $J$-holomorphic subvariety $\Theta_J$ representing $[S]$ if $J\in \mathcal J^{\omega}$. If we choose $J$ from $\mathcal J^{[S]}_{reg}$, which means any $J$-holomorphic map in class $[S]$ is regular, this representative is a smooth rational curve.

Since $\mathcal J^{[S]}_{reg}$ is also path connected, any $\Theta_J$ would be symplectically isotopic to $\Theta_{J'}$ when both $J, J'\in \mathcal J^{[S]}_{reg}$. This follows the second statement.
\end{proof}

The same result  does not hold for higher genus curves in rational surfaces: the hyperelliptic branch loci of the examples of \cite{OS} provide symplectic surfaces not homologous to holomorphic ones in rational surfaces. On the other hand, a fairly general construction of homologous nonisotopic tori in nonrational $4$-manifolds has been given by Fintushel and Stern \cite{FS}, and later by many others. However, our discussion in Section \ref{tori} and the argument we provide above could lead to some positive results on symplectic isotopy of tori. 



\begin{thebibliography}{99}
\bibitem{AZ} A. Akhmedov, W. Zhang, \textit{The fundamental group of symplectic $4$-manifolds with $b^+=1$},  arXiv:1506.08367.


\bibitem{B} J. Barraud, \textit{Nodal symplectic spheres in $\mathbb{CP}^2$ with positive self-intersection},  Internat. Math. Res. Notices 1999, no. 9, 495--508. 

\bibitem{CP} P. Cascini, D. Panov, \textit{Symplectic generic complex structures on 4-manifolds with $b_+ = 1$},  J. Symplectic Geom. 10 (2012), no. 4, 493--502.

\bibitem{CZ2} H. Chen, W. Zhang, \textit{Kodaira dimensions of almost complex manifolds II}, arXiv:2004.12825.

\bibitem{DL} J. Dorfmeister, T.J. Li, \textit{The relative symplectic cone and $T^2$-fibrations}, J. Symplectic Geom. 8 (2010), no. 1, 1--35.

\bibitem{DLW} J. Dorfmeister, T.J. Li, W. Wu, \textit{Stability and existence of surfaces in symplectic $4$-manifolds with $b^+=1$},  J. Reine Angew. Math. 742 (2018), 115--155.

\bibitem{DLZ} T. Draghici, T.J. Li, W. Zhang, \textit{Symplectic forms and cohomology decomposition of almost complex $4$-manifolds}, Int. Math. Res. Not. IMRN 2010, no. 1, 1--17.

\bibitem{FS} R. Fintushel, R. Stern, \textit{Symplectic surfaces in a fixed homology class}, J. Differential Geom. 52 (1999), 203--222.

 

\bibitem{Gr} M. Gromov, \textit{Pseudoholomorphic curves in symplectic manifolds},  Invent. Math.  82  (1985),  no. 2, 307--347.


\bibitem{hvle} H.-V. Le, \textit{Realizing homology classes by symplectic submanifolds}, preprint 2014, arXiv:math/0505562v3.

\bibitem{LBL} B.-H. Li, T.J. Li, \textit{Symplectic genus, minimal genus and diffeomorphisms}, Asian J. Math. 6 (2002), no. 1, 123--144.

\bibitem{LLwall} T.J. Li, A.-K. Liu, \textit{General wall crossing formula}, Math. Res. Lett. 2 (1995), no. 6, 797--810.
\bibitem{LLb+1}  T.J. Li, A.-K. Liu, \textit{The equivalence between SW and Gr in the case where $b^+=1$}, Internat. Math. Res. Notices 1999, no. 7, 335--345.


\bibitem{LL} T.J. Li, A.-K. Liu, \textit{Uniqueness of symplectic canonical class,
surface
cone and symplectic cone of 4-manifolds with $b\sp +=1$}, J.
Differential Geom. 58 (2001), no. 2, 331--370.

\bibitem{LW}T.J. Li, W. Wu, \textit{Lagrangian spheres, symplectic surfaces and the symplectic mapping class group}, Geom. Topol. 16 (2012), no. 2, 1121--1169.

\bibitem{LZ} T.J. Li, W. Zhang, \textit{Comparing tamed and compatible
symplectic cones and cohomological properties of almost complex manifolds},
 Comm. Anal. Geom. 17 (2009), no. 4, 651--683.

\bibitem{LZ-generic} T.J. Li, W. Zhang, \textit{Almost K\"ahler forms on rational $4$-manifolds}, Amer. J. Math. 137 (2015), no. 5, 1209--1256.

\bibitem{LZrc} T.J. Li, W. Zhang, \textit{$J$-holomorphic curves in a nef class}, Int. Math. Res. Not. (2015) Vol. 2015 12070--12104.

\bibitem{LP} N. Lindsay, D. Panov, \textit{$S^1$-invariant symplectic hypersurfaces in dimension $6$ and the Fano condition},  J. Topol. 12 (2019), no. 1, 221--285.

\bibitem{Loo} E. Looijenga, \textit{Rational surfaces with an anticanonical cycle}, 
Ann. of Math. (2) 114 (1981), no. 2, 267--322. 

\bibitem{Man} Yu. I. Manin, \textit{Cubic forms. Algebra, geometry, arithmetic}, Translated from the Russian by M. Hazewinkel. Second edition. North-Holland Mathematical Library, 4. North-Holland Publishing Co., Amsterdam, 1986. x+326 pp.


\bibitem{McD}  D. McDuff, \textit{The structure of rational and ruled symplectic
$4$-manifolds}, J. Amer. Math. Soc.  3 (1990), 679--712.


\bibitem{Mc2} D. McDuff,  \textit{Symplectomorphism groups and almost complex structures}, Essays on geometry and related topics, Vol. 1, 2, 527-556, Monogr. Enseign. Math., 38, Enseignement Math., Geneva, 2001.

\bibitem{MS} D. McDuff, D. Salamon, \textit{$J-$holomorphic curves and symplectic topology}, American Mathematical Society Colloquium Publications, 52. American Mathematical Society, Providence, RI, 2004.

\bibitem{MW} M. Micallef, B. White, \textit{The structure of branch points in minimal surfaces and in pseudoholomorphic curves},
Ann. of Math. (2) 141 (1995), no. 1, 35--85. 

\bibitem{OS} B. Ozbagci, A. Stipsicz, \textit{Noncomplex smooth 4-manifolds with genus-2 Lefschetz fibrations}, Proc. Amer. Math. Soc. 128 (2000), 3125--3128. 

\bibitem{Shev} V. Shevchishin, \textit{Pseudoholomorphic curves and the symplectic isotopy problem}, preprint; math.SG/0010262. 

\bibitem{ST} B. Siebert, G. Tian, \textit{On the holomorphicity of genus two Lefschetz fibrations},  Ann. of Math. (2) 161 (2005), no. 2, 959--1020. 

\bibitem{Sik} J.-C.  Sikorav, \textit{The gluing construction for normally generic $J-$holomorphic curves}, Symplectic and contact topology: interactions and perspectives (Toronto, ON/Montreal, QC, 2001), 175--199, Fields Inst. Commun., 35, Amer. Math. Soc., Providence, RI, 2003.

\bibitem{T} C. Taubes, \textit{${\rm SW}\Rightarrow{\rm Gr}$: from the Seiberg-Witten equations to pseudo-holomorphic curves},  J. Amer. Math. Soc.  9  (1996),  no. 3, 845--918.


\bibitem{T1} C. Taubes, \textit{Tamed to compatible: Symplectic forms via
moduli space integration}, J. Symplectic Geom. 9 (2011), 161--250.

\bibitem{p=h} W. Zhang, \textit{The curve cone of almost complex $4$-manifolds}, Proc. Lond. Math. Soc. (3) 115 (2017), no. 6, 1227--1275.

\bibitem{zuc} S. Zucker,  \textit{The Hodge conjecture for cubic fourfolds}, Compositio Math. 34 (1977), no. 2, 199--209.
\end{thebibliography}
\end{document}